\documentclass[12pt]{amsart}

\usepackage[margin = 1in]{geometry}

\usepackage{tikz}
\usepackage{tikz-cd}
\usepackage{url, hyperref}
\usepackage{float}
\usepackage{mathtools}

\tikzset{->-/.style={decoration={
  markings,
  mark=at position .45 with {\arrow{>}}},postaction={decorate}}}
  
\usetikzlibrary{shapes.arrows}
\usetikzlibrary{decorations.pathreplacing}
\usepackage[all]{xy}

\tikzset{->-/.style={decoration={
  markings,
  mark=at position .45 with {\arrow{>}}},postaction={decorate}}}

\usepackage{amsmath}
\usepackage{amssymb}
\usepackage{enumitem}
\usepackage{graphicx}
\usepackage{mathdots}
\usepackage{color}
\usepackage{diagbox}
\usepackage{array, makecell}
\usepackage{rotating}
\usepackage{amsmath}
\usepackage{tikz-cd}
\usetikzlibrary{calc, decorations.pathmorphing, shapes.misc}

\newcount\colveccount
\newcommand*\colvec[1]{
        \global\colveccount#1
        \begin{pmatrix}
        \colvecnext
}
\def\colvecnext#1{
        #1
        \global\advance\colveccount-1
        \ifnum\colveccount>0
                \\
                \expandafter\colvecnext
        \else
                \end{pmatrix}
        \fi
}

\usepackage{amsthm}

\theoremstyle{definition}
\newtheorem{definition}{Definition}[subsection]
\newtheorem{theorem}[definition]{Theorem}

\newtheorem{proposition}[definition]{Proposition}
\newtheorem{situation}[definition]{Situation}

\newtheorem{construction}[definition]{Construction}

\newtheorem{corollary}[definition]{Corollary}
\newtheorem{lemma}[definition]{Lemma}

\newtheorem{remark}[definition]{Remark}

\newtheorem{question}[definition]{Question}
\newtheorem{notation}[definition]{Notation}

\def\bK{\mathbb{K}}

\def\bP{\mathbb{P}}

\def\bZ{\mathbb{Z}}

\def\cC{\mathcal{C}}

\def\cE{\mathcal{E}}
\def\cF{\mathcal{F}}

\def\cH{\mathcal{H}}

\def\cL{\mathcal{L}}
\def\cM{\mathcal{M}}

\def\cO{\mathcal{O}}

\def\cR{\mathcal{R}}

\def\E{\mathrm{E}}
\def\H{\mathrm{H}}

\DeclareMathOperator{\Bl}{Bl}
\DeclareMathOperator{\charac}{char}

\DeclareMathOperator{\ev}{ev}

\DeclareMathOperator{\Gr}{Gr}

\DeclareMathOperator{\Inc}{Inc}

\DeclareMathOperator{\pr}{pr}

\DeclareMathOperator{\rank}{rank}
\DeclareMathOperator{\rel}{rel}

\DeclareMathOperator{\rk}{rk}

\DeclareMathOperator{\Tev}{Tev}

\DeclareMathOperator{\univ}{univ}

\title{Interpolation for rational curves with secants}

\author{Alessio Cela}
\address{ University of Cambridge, Department of pure mathematics and mathematical statistics
\hfill \newline\texttt{}
 \indent Centre for Mathematical Sciences, Wilberforce Road Cambridge, UK} \email{{\tt ac2758@cam.ac.uk}}

 \author{Carl Lian}
\address{Washington University in St. Louis, Department of Mathematics, 1 Brookings Drive
\hfill \newline\texttt{}
 \indent  St. Louis, MO 63130} \email{{\tt clian@wustl.edu}}

\date{\today}

\usepackage{graphicx}

\begin{document}

\maketitle

\begin{abstract}
In arbitrary characteristic, we determine the maximum number of general points through which a rational curve of degree $d$ in $\bP^r$ passes, subject to an additional secancy condition along a linear space. We consider the cases both where the points on the curve are unprescribed and prescribed, which amount to the determination of the normal and restricted tangent bundles of a general rational curve in $\Bl_{\bP^s}(\bP^r)$, respectively. In the appendix, we enumerate the interpolating curves in the case of prescribed points on the curve.

\end{abstract}

\setcounter{tocdepth}{1}

\tableofcontents

\section{Introduction}

\subsection{Interpolation with secants}

Let \(\bK\) be an algebraically closed field of arbitrary characteristic. In this paper, we address the following questions.

\begin{question}\label{q:interpolation}
Let \(d,k,r,s,n\) be integers with \(d > k \geq 0\) and \(r-2 \geq s \geq 0\). Let \(\bP^s \subset \bP^r\) be a linear space, and let \(x_1,\ldots,x_n \in \bP^r\) be general points. Write $C=\bP^1$.
\begin{enumerate}
\item[(a)] Does there exist a rational curve \(f \colon C \to \bP^r\) of degree \(d\), meeting \(\bP^s \subset \bP^r\) with multiplicity \(k\), and passing through \(x_1,\ldots,x_n\)?
\item[(b)] Let \(p_1,\ldots,p_n \in C\) be general points. Does there exist a rational curve \(f \colon C \to \bP^r\) of degree \(d\), meeting \(\bP^s \subset \bP^r\) with multiplicity \(k\), and such that \(f(p_i)=x_i\) for all \(i\)?
\end{enumerate}
\end{question}
When the set of maps $f$ is finite, it is also natural to ask for the number of $f$. One can more generally replace the points $x_i\in \bP^r$ with higher-dimensional subvarieties. These are instances of the question of \emph{interpolation}, which has been studied in different guises by many authors, see \cite{sacchiero,EV81, EV82,ramella, ran, atanasov,  larson,  aly,CR,Ran_rational_scrolls,  clv1,  lv, clv2,Ran-hypersurface, Mio2025} for an incomplete list.

A rational curve \(f \colon C \to \bP^r\) of degree \(d\) meeting \(\bP^s \subset \bP^r\) with multiplicity \(k\) is equivalently a map
\[
f \colon C \to X_{r,s} := \Bl_{\bP^s}(\bP^r)
\]
of degree \((d,k)\). By degree \((d,k)\), we mean that the degree of \(f\) against the hyperplane class, pulled back from \(\bP^r\), equals \(d\), and the degree of \(f\) against the exceptional divisor of \(X_{r,s}\) equals \(k\). We assume that \(d>k\) so that the composition of \(f\) with the projection \(\pi \colon X_{r,s} \to \bP^{r-s-1}\) is not constant.

In this language, Questions~\ref{q:interpolation}(a) and (b) ask whether the evaluation maps
\[
\begin{gathered}
\rho \colon \cM_{0,n}(X_{r,s},(d,k)) \to (X_{r,s})^n, \quad \text{and} \\
\tau \colon \cM_{0,n}(X_{r,s},(d,k)) \to M_{0,n} \times (X_{r,s})^n
\end{gathered}
\]
respectively, are dominant. \(\cM_{0,n}(X_{r,s},(d,k))\) denotes the moduli space of \(n\)-pointed maps from a rational curve to \(X_{r,s}\) of degree \((d,k)\), which is smooth and irreducible of dimension
\[
(n-3)+(r+1)d-(r-s-1)k+r.
\]
A necessary condition for $\rho$ (resp. $\tau$) to be dominant is for \(\dim(\cM_{0,n}(X_{r,s},(d,k)))\) to be at least $\dim((X_{r,s})^n)$ (resp. $\dim(M_{0,n} \times (X_{r,s})^n)$), but we will see that neither is sufficient in general (\S\ref{sec:obstructions}, Proposition \ref{prop:projection_split_tangent}).

\subsection{Normal and tangent bundles}\label{sec:bundles}

In fact, our main results answer the following more refined questions.

\begin{question}\label{q:bundles}
Let \(d,k,r,s\) be as in Question~\ref{q:interpolation}. Let \(f :C=\bP^1 \to X_{r,s}\) be a general curve of degree $(d,k)$.
\begin{enumerate}
\item[(a)] What is the normal bundle \(N_{C/X_{r,s}}\)?
\item[(b)] What is the restricted tangent bundle \(f^*T_{X_{r,s}}\)?
\end{enumerate}
\end{question}
We will see that \(f\) is unramified (\S\ref{sec:unramified}), so the normal bundle \(N_{C/X_{r,s}}\) is sensible. The bundles in question split as direct sums of line bundles:
\[
\begin{aligned}
N_{C/X_{r,s}} &\cong \cO(a_1) \oplus \cdots \oplus \cO(a_{r-1}), \\
f^*T_{X_{r,s}} &\cong \cO(b_1) \oplus \cdots \oplus \cO(b_r).
\end{aligned}
\]
Question~\ref{q:bundles} asks for the values of the integers \(a_j,b_j\). In characteristic 0, the splitting of the normal and restricted tangent bundles determine the answers to Question~\ref{q:interpolation}(a) and~\ref{q:interpolation}(b), respectively. 

Indeed, if $\charac(\bK)=0$, then the map 
\[
\rho \colon \cM_{0,n}(X_{r,s},(d,k)) \to (X_{r,s})^n
\] is dominant if and only if it is smooth at a general point $f\in \cM_{0,n}(X_{r,s},(d,k))$. The relative obstruction space at $f$ is $H^1(C,N_{C/X_{r,s}}(-p_1-\ldots-p_n))$, which vanishes if and only if $n\le \min(a_j)+1$. Therefore, the answer to Question~\ref{q:interpolation}(a) is affirmative if and only if $n\le \min(a_j)+1$. Similarly, the answer to Question~\ref{q:interpolation}(b) is affirmative if and only if $H^1(C,f^*T_{X_{r,s}}(-p_1-\ldots-p_n))=0$, if and only if $n\le \min(b_j)+1$.

We will determine the answers to Question~\ref{q:bundles} in arbitrary characteristic, thereby answering Question~\ref{q:interpolation} in characteristic 0. We will see furthermore that the answers to Question~\ref{q:interpolation} are uniform in the characteristic (Proposition \ref{prop:rho_char_indep}, Corollary \ref{cor:interpolation_tangent}), so are therefore determined for any $\bK$. We will find that the splitting type of $f^*T_{X_{r,s}}$ is also uniform in the characteristic, but the splitting type of $N_{C/X_{r,s}}$ is not. In particular, if $\charac(\bK)=2$, it is already the case when $k=0$ (and remains so when $k>0$) that $\rho$ can be dominant, but inseparable. In this case, the answer to Question~\ref{q:interpolation}(a) can be affirmative even if $n>\min(a_j)+1$.

From the point of view of interpolation, the nicest situation is when the vector bundles in question satisfy the following property.

\begin{definition}\label{def:balanced}
We say that a vector bundle $\cE$ on $\bP^1$ is \emph{balanced} if $\cE\cong \cO(a)^x\oplus \cO(a+1)^y$ for some integers $a,x,y$, possibly 0.
\end{definition}

Up to isomorphism, there is a unique balanced vector bundle on $\bP^1$ of any given degree and rank. If $H^1(\cE)=0$, then $\cE$ is balanced if and only if it \emph{satisfies interpolation}, in the sense of \cite[\S 4.1]{aly}, see Definition \ref{def:interpolation_aly}.  By the discussion above, in characteristic 0, $N_{C/X_{r,s}}$ (resp. $f^*T_{X_{r,s}}$) is balanced if and only if $\rho$ (resp. $\tau$) is dominant for the largest possible value of $n$ allowed by the constraint that the dimension of the target be at most that of the source. See \cite[\S 3, \S 8]{atanasov} for a detailed discussion of different notions of interpolation of vector bundles on curves, including restrictions that amount to imposing incidence conditions at arbitrary linear spaces.

While our results are more general, we will focus in the remainder of the introduction on the question of when $N_{C/X_{r,s}}$ and $f^*T_{X_{r,s}}$ are balanced.

\subsection{Projective spaces}

The case $k=0$ corresponds to rational curves in $\bP^r$ with no secancy condition. We review what is known.

\begin{theorem}\label{thm:Pr}\cite{sacchiero,ran,aly,lv,clv2}
Let $f:C=\bP^1\to \bP^r$ be a general rational curve of degree $d$, where $d>0$ and $r\ge 2$.
\begin{enumerate}
\item[(a)] If $d \geq r$ and $\charac(\bK) \neq 2$, then $N_{C/\bP^r}$ is balanced.
\item[(b)] Suppose $d < r$ or $\charac(\bK) = 2$, and write $\ell\in [0,r-2]$ for the residue of $d-1 \pmod{r-1}$. Then,
\[
N_{C/\bP^r}\cong \cO(\alpha)^{r-1-\ell} \oplus \cO(\alpha+2)^\ell,
\]
where
\[
\alpha = d+2 \cdot \frac{d-1-\ell}{r-1} = d+2 \cdot \left\lfloor \frac{d-1}{r-1} \right\rfloor.
\]
In particular, under these assumptions, $N_{C/\bP^r}$ is balanced if and only if $(r-1) \mid (d-1)$.
\end{enumerate}
\end{theorem}

A map \(f \colon C \to \bP^r\) is \emph{non-degenerate} if its image is not contained in any hyperplane. If \(d<r\), then a general map of degree \(d\) is degenerate, factoring through a rational normal curve in \(\bP^d\):
\[
f \colon C \to \bP^d \hookrightarrow \bP^r.
\]
Then, we have a short exact sequence
\[
0 \to N_{C/\bP^d} \to N_{C/\bP^r} \to N_{\bP^d/\bP^r}\vert_{C} \to 0
\]
inducing a splitting $N_{C/\bP^r}\cong N_{\bP^d/\bP^r}\vert_{C}\oplus N_{C/\bP^d}\cong \cO(d)^{r-d}\oplus \cO(d+2)^{d-1}$. Thus, $N_{C/\bP^r}$ is not balanced unless $d=1$. If \(\charac(\bK) \neq 2\), then non-degeneracy is the only obstruction to balancedness of \(N_{C/\bP^r}\).

In characteristic \(2\), there is an additional obstruction, namely that \(N_{C/\bP^r}^\vee(1)\) is isomorphic to the pullback of a vector bundle under Frobenius~\cite[\S 2.2]{lv}. As a result, all summands of \(N_{C/\bP^r}\) must have degree of the same parity as $d$, which obstructs balancedness unless $(r-1) \mid (d-1)$. Nevertheless, it is true \emph{independently of characteristic} that, for any $d\geq r$, there exists a rational curve of degree $d$ in $\bP^r$ passing through
\[
n=\left\lfloor \frac{(r+1)d+r-3}{r-1}\right\rfloor = d + \left\lfloor \frac{2d-2}{r-1}\right\rfloor+1
\]
general points \cite[Theorem 1.2, Lemma 9.1]{lv}. When $n=\alpha+2$ and $\ell\ge\frac{r-1}{2}$, the map 
\[
\rho:\cM_{0,n}(\bP^r,d)\to (\bP^r)^n
\]
is dominant, but inseparable.

The situation is simpler for restricted tangent bundles.

\begin{theorem}\label{thm:Pr_tangent}\cite{ramella,larson}
Let \(f \colon C=\bP^1 \to \bP^r\) be a general rational curve of degree \(d\), where \(d>0\) and \(r \geq 2\). Then, \(f^*T_{\bP^r}\) is balanced, independently of \(\charac (\bK)\).
\end{theorem}

We remark that the more recent works establish interpolation for tangent bundles \cite{larson} and normal bundles (with an explicit, finite list of exceptions) \cite{lv} of general \emph{Brill--Noether curves} of positive genus in \(\bP^r\). Analogous questions for general curves of positive genus in \(X_{r,s}\) are sensible and interesting, due to Farkas's analog of the Brill-Noether theorem for general curves with secants \cite[Theorem 0.1]{farkas} (see also \cite[Corollary 2.3.4]{cl_complete}), but we do not consider them in this paper.

\subsection{Obstructions to balancedness of normal bundles}

In this section, we describe our new results.

\begin{definition}\label{def:degenerate}
We say that $f:C\to X_{r,s}$ is \emph{non-degenerate} if it is so after post-composing with both maps $b:X_{r,s}\to \bP^r$ and $\pi:X_{r,s}\to \bP^{r-s-1}$. We say it is \emph{degenerate} otherwise. 
\end{definition}

A general $f$ of degree $(d,k)$ is non-degenerate if $d\ge r$ and $d-k\ge r-s-1$. We will also say that the tuple $(d,k;r,s)$ (or $(d,k)$, when $r,s$ are implicit) is non-degenerate (resp. degenerate) if these inequalities hold (resp. fail). Equivalently, \(f\) is degenerate if its image factors through the proper transform of a linear space \(Y \subset X_{r,s}\). As in the case of projective spaces, degeneracy can obstruct balancedness of \(N_{C/X_{r,s}}\), see \S\ref{sec:degenerate}.

However, new obstructions arise in our setting. If $s\neq r-2$, then the projection \(\pi \colon X_{r,s} \to \bP^{r-s-1}\) induces a short exact sequence
\begin{equation}\label{eq:proj_seq}
0\to T_{X_{r,s}/\bP^{r-s-1}}|_{C} \to N_{C/X_{r,s}}\to N_{C/\bP^{r-s-1}} \to 0.
\end{equation}

When there is no opportunity for confusion, we write $T_{\rel}:=T_{X_{r,s}/\bP^{r-s-1}}|_{C}$. We also write $N_{(d,k;r,s)}:=N_{C/X_{r,s}}$, or $N_{d,k}$ when $r,s$ are implicit. 

We will see that \(T_{\rel}\) is always balanced, Lemma~\ref{lem:Trel_balanced}. Therefore, whenever \(N_{C/\bP^{r-s-1}}\) is balanced and has sufficiently small slope compared to that of \(T_{\rel}\), the exact sequence \eqref{eq:proj_seq} splits, and induces an unbalanced splitting of $N_{C/X_{r,s}}$. See Proposition \ref{prop:projection_split_char_not2} for a precise statement. We will show that, for non-degenerate curves in characteristic other than 2, the splitting of \eqref{eq:proj_seq} is the \emph{only} obstruction to balancedness of $N_{d,k}$.

\begin{theorem}\label{thm:intro_char_not2}
Suppose that $\charac(\bK)\neq2$ and that $(d,k)$ is non-degenerate. If $s\neq r-2$ and
\[
d < \frac{(r-s)(k-1)+4+(r-s-2)\left\lceil \frac{k}{s+1} \right\rceil}{2} =: L_{r,s}(k),
\]
then the exact sequence \eqref{eq:proj_seq} splits, and $N_{d,k}$ is not balanced. Otherwise, $N_{d,k}$ is balanced. In particular, if $s=r-2$, then $N_{d,k}$ is always balanced.
\end{theorem}

Thus, when \(\charac(\bK) \neq 2\) and \((d,k)\) is non-degenerate, \(N_{d,k}\) is determined. The degenerate cases are reduced to the non-degenerate ones in Proposition \ref{prop:degenerate_splittings}, see also Theorem \ref{thm:main_char_not2}.

In characteristic \(2\), the situation is more subtle, because \(N_{C/\bP^{r-s-1}}\) is less-often balanced. It is still the case that, if \(d\) is small compared to \(k\), then \eqref{eq:proj_seq} splits and obstructs the balancedness of \(N_{d,k}\), see Proposition \ref{prop:projection_split_char2}. However, it may also happen that \(N_{C/\bP^{r-s-1}}\) has too many summands of small degree to admit a surjection from a balanced $N_{d,k}$, without \eqref{eq:proj_seq} splitting for slope reasons. Nevertheless, this failure also occurs only when $d$ is small compared to \(k\), see Proposition \ref{prop:projection_char2_low_obstruction}.

A different source of failure of balancedness in characteristic \(2\) is the propagation of unbalancedness from curves in $\bP^r$. Namely, we have a short exact sequence
\begin{equation}\label{eq:subsheaf}
0 \to N_{d,k} \to N_{C/\bP^r} \to \cF \to 0,
\end{equation}
where \(\cF\) is a torsion sheaf supported at the intersection of \(C\) with the exceptional divisor of \(X_{r,s}\). Indeed, deformations of $C$ in $X_{r,s}$ are equivalently deformations of $C$ inside $\bP^r$ which continue to meet $\bP^s$ with multiplicity $k$.

If $(r-1)\nmid (d-1)$, then $N_{C/\bP^r}$ is not balanced, by Theorem \ref{thm:Pr}. If $k$ is small, then $\cF$ has degree too small for the subsheaf $N_{d,k}$ to be balanced, see Corollary \ref{cor:char2_obstruction} for a precise statement. Moreover, if $s=0$, then $N_{d,k}\cong N_{C/\bP^r}(-(C\cap \bP^s))$, and can \emph{never} be balanced if $(r-1)\nmid (d-1)$. Our main result for normal bundles in characteristic 2 is that these are the only obstructions to balancedness.

\begin{theorem}\label{thm:intro_char_2}
Suppose that $\charac(\bK)=2$. If $N_{d,k}$ fails to be balanced, then it is explained either by \eqref{eq:proj_seq} or \eqref{eq:subsheaf}. More precisely, the unique balanced vector bundle on $\bP^1$ of degree and rank equal to that of $N_{d,k}$ cannot fit into both exact sequences \eqref{eq:proj_seq},\eqref{eq:subsheaf}.
\end{theorem}

In contrast to Theorem \ref{thm:intro_char_not2}, Theorem \ref{thm:intro_char_2} does not require restricting to non-degenerate $(d,k)$. The exact splitting types of $N_{d,k}$ in characteristic 2 are determined in Theorems \ref{thm:d>k(r-s-1)}, \ref{thm:d<k(r-s-1)_epsilon_not0}, and \ref{thm:d<k(r-s-1)_epsilon_zero}.

\subsection{Tangent bundles}\label{sec:intro_tangent}

For any $r,s$ (including $s=r-2$), the projection \(\pi \colon X_{r,s} \to \bP^{r-s-1}\) induces a short exact sequence
\begin{equation}\label{eq:proj_seq_tangent}
0\to T_{\rel} \to f^{*}T_{X_{r,s}} \to f^{*}\pi^{*}T_{\bP^{r-s-1}} \to 0.
\end{equation}
Write $T_{d,k}$ or $T_{(d,k;r,s)}$ for $f^{*}T_{X_{r,s}}$, and $T_{\bP^{r-s-1}}|_{C}$ for $f^{*}\pi^{*}T_{\bP^{r-s-1}}$. Similarly to the case of normal bundles, if the slope of $T_{\bP^{r-s-1}}|_{C}$ is too much smaller than that of $T_{\rel}$, then $f^{*}T_{X_{r,s}}$ cannot be balanced, see Proposition \ref{prop:projection_split_tangent}. This gives rise to the only obstruction to balancedness of the restricted tangent bundle.

\begin{theorem}\label{thm:intro_tangent}
If
\[
d < (r-s)(k-1)+(r-s-1)\left\lceil\frac{k}{s+1}\right\rceil+1=:L^t_{r,s}(k),
\]
then the exact sequence \eqref{eq:proj_seq_tangent} splits, and $T_{d,k}$ is not balanced. Otherwise, $T_{d,k}$ is balanced.
\end{theorem}

Note that no hypotheses on non-degeneracy or characteristic are required. In particular, $T_{d,k}$ is determined in all cases. 

Recall that, in characteristic 0, the splitting type of $T_{d,k}$ also determines the maximum number of general incidence conditions $f(p_i)=x_i$ that a general map $f:C\to X_{r,s}$ can satisfy, Question \ref{q:interpolation}(b). In the appendix, we compute the \emph{number} of such interpolating maps, subject to incidence conditions $f(p_i)\in \Lambda_i$ at general linear spaces of arbitrary dimension (with possible non-reduced multiplicities in positive characteristic). Our enumeration is obtained from integration on the moduli space of quasimaps \cite{cfk}, and an analysis of its boundary geometry to rule out excess contributions. We recover as a special case the genus 0 ``Tevelev degrees'' of $X_{r,s}$, which by definition equal the following counts of maps.

\begin{theorem}\label{thm:tev}
Suppose that $n=\frac{r+1}{r}\cdot d-\frac{r-s-1}{r}\cdot k+1\in \bZ$. Then, the number of maps $f:\bP^1\to X_{r,s}$ satisfying $n$ general  incidence conditions $f(p_i)=x_i$ is 
\[
\Tev^{X_{r,s}}_{0,n,(d,k)}=\binom{(s+1)(n-d-1)}{k}
\]
if $(s+1)(n-d-1)\ge k$, and 0 otherwise.
\end{theorem}

Theorem \ref{thm:tev} had previously been obtained in  \cite{cl2} (when $s=0$) and \cite{lsakran} (in general), but the counts with higher-dimensional linear spaces are new, see Proposition \ref{prop:integral_value}. The condition on non-vanishing is equivalent to $d \ge L^t_{r,s}(k)$, so consistent with Theorem \ref{thm:intro_tangent}. In fact, one obtains a different proof of the balancedness of $T_{d,k}$ whenever $d \ge L^t_{r,s}(k)$ from the non-vanishing of the number of (pointed) maps interpolating through the maximum number of points and one additional linear space, whenever this quantity is not divisible by $\charac(\bK)$.

\subsection{Methods}

We prove our positive results for balancedness of vector bundles, Theorems \ref{thm:intro_char_not2}, \ref{thm:intro_char_2}, and \ref{thm:intro_tangent}, by degeneration. The property of interpolation is open in families of nodal curves, so it suffices to consider normal and tangent bundles of singular curves $f:C\cup L_1\cdots\cup L_m \to X_{r,s}$. 

For normal bundles, we take the $L_j$ to be a collection of $r-s-2$ general lines of degree $(1,0)$ and one of degree $(1,1)$. This allows us to deduce balancedness of $N_{d+(r-s-1),k+1}$ from that of $N_{d,k}$ (Proposition \ref{prop:main_inductive}). The existence of many exceptions to balancedness poses difficulties, but in fact the degeneration is sufficiently robust to deduce balancing of $N_{d+(r-s-1),k+1}$ even if $N_{d,k}$ is ``close'' to being balanced. Combining this degeneration with an analysis of the projection sequence \eqref{eq:proj_seq} is almost enough to determine all normal bundles $N_{d,k}$, taking Theorem \ref{thm:Pr} as input. To cover the remaining cases, we borrow some additional degeneration arguments from \cite{lv}, see Proposition \ref{prop:lv_inductive}.

The case of tangent bundles is easier, requiring one line $L_1$ of degree $(1,0)$ for the degeneration argument, as in \cite{larson}. The projection sequence \eqref{eq:proj_seq_tangent} provides sufficiently many base cases.

\subsection{Plan of paper} After preliminaries in \S\ref{sec:prelim}, the main geometric input into the determination of normal bundles $N_{d,k}$ is given in \S\ref{sec:obstructions} and \S\ref{sec:inductive}. The results of \S\ref{sec:obstructions} ``bound unbalancedness from below,'' and those of \S\ref{sec:inductive} ``bound unbalancedness from above.'' The determination of $N_{d,k}$ is completed in \S\ref{sec:char_not2} (in characteristic not equal to 2) and \S\ref{sec:char2} (in characteristic 2). In particular, the answer to Question \ref{q:interpolation}(a) is given in \S\ref{sec:interpolation_normal}. We deal with restricted tangent bundles $T_{d,k}$ and Question \ref{q:interpolation}(b) in \S\ref{sec: interpolation_tangent}. Our enumerative calculations, which involve somewhat different methods, appear in the appendix. 

\section*{Acknowledgements}
Portions of this work were carried out during the second author's visits to the University of Cambridge and during both authors' participation in the 2025 Summer Research Institute in Algebraic Geometry at Colorado State University. We are grateful to these institutions for their hospitality. We thank Gavril Farkas, Eric Jovinelly, Eric Larson, Rahul Pandharipande, Dhruv Ranganathan, and Isabel Vogt for conversations related to this work.

A.~C.\ was supported by SNF grant P500PT-222363. C.~L.\ has been supported by NSF Postdoctoral Fellowship DMS-2001976 and by an AMS--Simons Travel Grant. 

\section{Preliminaries}\label{sec:prelim}

In this section, we show that a general degree $(d,k)$ map $f \colon C=\bP^1 \to X_{r,s}$ is unramified, so that the (lci) normal bundle
$
N_f = N_{C/X_{r,s}} = N_{d,k}
$
is well defined. We then review several notions and results concerning vector bundles from \cite{aly, lv} that will be used throughout the paper.

\subsection{The general map is unramified}\label{sec:unramified}

\begin{proposition}\label{prop:unramified_map}
Fix $(d,k)$ with $d>k \geq 0$. Then, the general map $f \colon C=\bP^1 \to X_{r,s}$ of degree $(d,k)$ is unramified.  
\end{proposition}

The moduli space of maps \(f: \bP^1 \to X_{r,s}\) of degree \((d,k)\) admits a dominant rational map from an affine space parametrizing sections of line bundles on $\bP^1$ defining \(f\), so is in particular irreducible. Therefore, the notion of a general map of degree \((d,k)\) is sensible.

\begin{lemma}\label{lem:free}
Let \(f: \bP^1 \to X_{r,s}\) be any map of degree \((d,k)\). For every $p \in \bP^1$, we have $H^1(\bP^1,f^*T_{X_{r,s}}(-2p))=0$. Equivalently, every summand of $f^*T_{X_{r,s}}$ has degree at least 1.
\end{lemma}

\begin{proof}
    We have an Euler sequence \cite[Theorem 8.1.6]{cls},
\[
0 \to \cO_{X_{r,s}}^2 \to \bigoplus_{\rho} \cO_{X_{r,s}}(D_\rho) \to T_{X_{r,s}} \to 0,
\]
where the \(D_\rho \subset X_{r,s}\) are the torus-invariant divisors on \(X_{r,s}\), of degrees \(k\), \(d-k\), \(d\) upon pullback to $\bP^1$. If $k>0$, then twisting and applying the long exact sequence in cohomology yields \(H^1(\bP^1, f^* T_{X_{r,s}}(-2p)) = 0\), because we also have \(d, d-k > 0\). If \(k = 0\), then regard \(f\) instead as a map \(f:\bP^1\to \bP^r\), and apply instead the usual Euler sequence for \(T_{\bP^r}\).
\end{proof}

Now, let $p \in \bP^1$ be a point, and let $\mathfrak{m}_p \subset \mathcal{O}_{\bP^1}$ be its ideal sheaf. Recall that the first jet bundle fiber of a vector bundle $\cE$ at $p$ is
\begin{equation*}
    J^1(\cE)_p = \cE \otimes_{\mathcal{O}_{\bP^1}} (\mathcal{O}_{\bP^1} / \mathfrak{m}_p^2).
\end{equation*}

\begin{lemma}\label{lemma:ev_is_surjective}
    Let \(f: \bP^1 \to X_{r,s}\) be any map of degree \((d,k)\). For every $p \in \bP^1$,  the evaluation map 
    \[
    \text{ev}^1_p \colon H^0(\bP^1, f^*T_{X_{r,s}}) \to J^1(f^*T_{X_{r,s}})_p
    \]
    is surjective.
\end{lemma}

\begin{proof}

Tensoring the short exact sequence 
\[
0 \to \mathfrak{m}_p^2 \to \mathcal{O}_{C} \to \mathcal{O}_{C} / \mathfrak{m}_p^2 \to 0
\]
with $f^*T_{X_{r,s}}$ yields
\[
    0 \to f^*T_{X_{r,s}}(-2p) \to f^*T_{X_{r,s}} \to f^*T_{X_{r,s}} \otimes (\mathcal{O}_{C} / \mathfrak{m}_p^2) \to 0.
\]
Now, apply the long exact sequence in cohomology and Lemma \ref{lem:free}.
\end{proof}

\begin{proof}[Proof of Proposition \ref{prop:unramified_map}]
Let \(\cM\) be the moduli space of maps \(f \colon \bP^1 \to X_{r,s}\) of degree \((d,k)\). Then, \(\cM\) is smooth of the expected dimension, because, by Lemma \ref{lem:free}, we have \(H^1(\bP^1, f^*T_{X_{r,s}}) = 0\) at every point. Consider the universal curve \(\cC = \cM \times \bP^1\), with the universal evaluation map \(F \colon \cC \to X\). The relative differential along the \(\bP^1\) factor defines a universal global section \(\sigma\) of the vector bundle
\[
\cH = F^*T_{X_{r,s}} \otimes T_{\bP^1}^{\vee}.
\]
A map \(f\) is ramified at \(p\) if and only if \(\sigma([f],p) = 0\), which means \(df_p = 0\).

Let \(y = ([f], p) \in \cC\) be a point at which \(\sigma\) vanishes. We aim to show that \(\sigma\) is transverse to the zero section at \(y\). If so, then the universal ramification locus \(\cR = Z(\sigma) \subset \cC\) is a subvariety of the expected codimension equal to \(\rank(\cH) = r\), and its image under the projection \(\cC \to \cM\) has codimension at least \(r-1\ge 1\). In particular, a general map \(f \colon \bP^1 \to X_{r,s}\) has no ramification points.

To prove the required transversality statement, we need to study the differential of \(\sigma\). The tangent space to \(\cC\) at \(y\) is
\[
T_y\cC \cong H^0(\bP^1, f^*T_X) \oplus T_p \bP^1,
\]
and the derivative of \(\sigma\) at \(y\) is given by
\begin{equation}\label{eqn:differential_sigma}
d\sigma_y \colon H^0(\bP^1, f^*T_X) \oplus T_p \bP^1 \to \cH_y \cong T_{f(p)}X \otimes T_p^\vee \bP^1.
\end{equation}
Under the identification
\[
T_{f(p)}X \otimes T_p^\vee \bP^1 \cong T_{f(p)}X \otimes \mathfrak m_p/\mathfrak m_p^2 \subseteq J^1(f^*T_{X_{r,s}})_p
\]
with the subset of jets vanishing at \(p\), the restriction of the differential \(d\sigma_y\) in \eqref{eqn:differential_sigma} to \(H^0(\bP^1, f^*T_{X_{r,s}}(-p)) \times \{0\}\) is precisely the restriction of the evaluation map
\[
\ev_p \colon H^0(\bP^1, f^*T_{X_{r,s}}(-p)) \to T_{f(p)}X \otimes \mathfrak m_p/\mathfrak m_p^2,
\]
which is surjective by Lemma \ref{lemma:ev_is_surjective}.
\end{proof}

\subsection{Generalities on interpolation}

\begin{definition}\cite[Definition 4.1]{aly}\label{def:interpolation_aly}
    Let $C$ be a curve (possibly singular, of any genus), and let $\mathcal{E}$ be a vector bundle of rank $m$ on $C$. A subspace of sections $V \subseteq H^0(C,\mathcal{E})$ \emph{satisfies interpolation} if $H^1(C,\mathcal{E})=0$ and, for every $j \ge 1$, there exists a collection of points $p_1, \ldots, p_j \in C^{\mathrm{sm}}$ such that
\[
\operatorname{dim} \bigg( V \cap H^0(C,\mathcal{E}(-p_1-\ldots-p_j)) \bigg) = \max(0, \operatorname{dim} V - mj).
\]

We say that $\mathcal{E}$ satisfies interpolation if $V = H^0(C,\mathcal{E})$ does. Equivalently (\cite[Proposition 4.5]{aly}), $\mathcal{E}$ satisfies interpolation if $H^1(C,\mathcal{E})=0$ and, for every integer $\delta\ge 1$, there exists an effective divisor $\Delta\subset C^{\mathrm{sm}}$ of degree $\delta$ for which $h^0(\cE(-\Delta))=0$ \emph{or} $h^1(\cE(-\Delta))=0$.
\end{definition} 

If $C=\bP^1$ and $H^1(C,\mathcal{E})=0$, then $\mathcal{E}$ satisfies interpolation if and only if it is balanced (Definition \ref{def:balanced}). Because all of the vector bundles on $\bP^1$ that we study in this paper will have $H^1(\bP^1,\mathcal{E})=0$, we will often use these terms interchangeably. It follows from semi-contininuity that interpolation is an open condition in flat families of vector bundles.

\begin{lemma}[{\cite[Lemma 2.5]{larson}}]\label{lem:interpolation_subspace}
Let $C$ be an irreducible curve, let $\mathcal{E}$ be a vector bundle on $C$. Let $p \in C_{\mathrm{sm}}$ be a general point, and let $\Lambda \subseteq E|_p$ be a general linear subspace. Suppose that $\mathcal{E}$ satisfies interpolation. Then,
\[
\{\sigma \in H^0(\mathcal{E}) \colon \sigma(p) \in \Lambda\} \subseteq H^0(\mathcal{E})
\]
satisfies interpolation, and has dimension $\max\{0, \chi(\mathcal{E}) - \operatorname{codim} \Lambda\}$.
\end{lemma}

The next lemma explains how to deduce interpolation for a reducible curve from that of its components. When combined with Lemma \ref{lem:maps_smooth} on smoothability of maps out of nodal curves, its strength lies in the fact that interpolation is an open condition in flat families.

\begin{lemma}[{\cite[Proposition 8.1]{aly}}]\label{lem: criterion interpolation}
Let $C=Z\cup Y$ be a reducible curve, and let $\mathcal{E}$ be a vector bundle on $C$. Let $D$ be an effective divisor on $C$ disjoint from $Z \cap Y$. Assume that
\[
H^0\big(\mathcal{E}|_Z(-D - Z \cap Y)\big) = 0.
\]
Let
\begin{align*}
    \mathrm{ev}_Z \colon H^0(\mathcal{E}|_Z) &\longrightarrow H^0(\mathcal{E}|_{Z \cap Y}) \\
    \mathrm{ev}_Y \colon H^0(\mathcal{E}|_Y) &\longrightarrow H^0(\mathcal{E}|_{Z \cap Y})
\end{align*}
denote the evaluation morphisms. Suppose that
\[
V = \mathrm{ev}_Y^{-1}\big(\mathrm{ev}_Z(H^0(\mathcal{E}|_Z(-D)))\big) \subseteq H^0(\mathcal{E}|_Y)
\]
satisfies interpolation, and has dimension
\[
\chi(\mathcal{E}|_Y) + \chi\big(\mathcal{E}|_Z(-D - Z \cap Y)\big).
\]
Then, $\cE$ satisfies interpolation.
\end{lemma}

\subsection{Pointing bundles and modifications of bundles}

\subsubsection{Pointing bundles}

Let \(C\) be a nodal curve, and let \(f \colon C \to X_{r,s}\) be an unramified map. Let \(E \subset X_{r,s}\) be the exceptional divisor, and \(\Lambda \subset X_{r,s}\) the strict transform of a linear subspace \(\overline{\Lambda} \subset \bP^r\). Assume that \(U = C \smallsetminus (f^{-1}(E) \cup f^{-1}(\Lambda))\) contains the singular locus of $C$.

In \cite[\S 5]{aly}, a subbundle \(N_{U \to \Lambda} \subseteq N_f|_U\) is constructed, whose fiber at a point of \(U\) consists of the normal directions ``pointing toward'' \(\overline{\Lambda}\). Because \(C^{\mathrm{sing}} \subseteq U\) and \(N_{U \to \Lambda} \subseteq N_f|_U\), the bundle \(N_{U \to \Lambda}\) extends to a subbundle of \(N_f\), which we call \(N_{C \to \Lambda}\).

\subsubsection{Modifications of vector bundles}\label{sec: pointing bundles}

\begin{definition}[{\cite[Definition 3.1]{lv}}]
Let $\mathcal{E}$ be a vector bundle on a scheme $X$. Let $D \subset X$ be a Cartier divisor, and let $\mathcal{F} \subset \mathcal{E}|_D$ be a subbundle of the restriction of $\mathcal{E}$ to $D$. Assume furthermore that $\mathcal{F}$ extends to a subbundle of $\mathcal{E}$ in an open neighborhood containing $D$. The \emph{negative elementary modification} of $\mathcal{E}$ along $D$ toward $\mathcal{F}$ is defined by
\begin{equation}
    \mathcal{E}[D \xrightarrow{-} \mathcal{F}] \coloneqq \ker \left( \mathcal{E} \longrightarrow \mathcal{E}|_D / \mathcal{F} \right).
\end{equation}
The \emph{positive elementary modification} of $\mathcal{E}$ along $D$ towards $\mathcal{F}$ is defined by
\begin{equation}
    \mathcal{E}[D \xrightarrow{+}  \mathcal{F}] \coloneqq \mathcal{E}[D \xrightarrow{-}  \mathcal{F}](D).
\end{equation}
\end{definition}

\begin{remark}
This notation agrees with that in \cite{lv, clv2}, but departs slightly from that of~\cite{aly}, where the notation $\mathcal{E}[D \to \mathcal{F}]$ was reserved exclusively for negative modifications. No positive modifications appear in \cite{aly}.
\end{remark}

It is useful to consider modifications simultaneously. 

\begin{definition}{\cite[Definition 2.16]{aly}}
    A \emph{modification datum} for a vector bundle $\mathcal{E}$ on $X$ is an ordered collection of triples
\[
M = \{(D_i, U_i, \mathcal{F}_i) \mid i \in I\}
\]
such that, for each $i$:
\begin{enumerate}
    \item[(a)] $D_i$ is an effective Cartier divisor on $X$,
    \item[(b)] $U_i \subset X$ is an open set containing the support of $D_i$, and
    \item[(c)] $\mathcal{F}_i \subset \mathcal{E}|_{U_i}$ is a subbundle.
\end{enumerate}
In addition, a datum $M$ is \emph{tree-like} if, for all $I' \subset I$ and $x \in X$, either:
\begin{enumerate}
    \item the set of subspaces $\{\mathcal{F}_i|_x \mid i \in I'\}$ is linearly independent in $\mathcal{E}_x$, or
    \item there are distinct indices $i, j \in I'$ and an open set $U \subset X$ containing $x \in X$ such that $\mathcal{F}_i|_U \subset \mathcal{F}_j|_U$.
\end{enumerate}
\end{definition}

Let $\mathcal{E}$ be a vector bundle on a variety $X$, let $\mathcal{F} \subset \mathcal{E}|_U$ be a subbundle, and let $D$ be an effective divisor on $X$ whose support is contained in $U$. Define
\[
S^{\mathrm{md}}(\mathcal{E}, \mathcal{F}, D) = \left\{ M = \{(D_i, U_i, \mathcal{F}_i)\} \;\middle|\; \{(D, U, \mathcal{F})\} \cup M \text{ is a tree-like modification datum} \right\}.
\]

In \cite[Proposition 2.17]{aly}, a bijection 
\begin{align*}
\varphi^{\mathrm{md}} : S^{\mathrm{md}}(\mathcal{E}, \mathcal{F}, D) &\longrightarrow S^{\mathrm{md}}(\mathcal{E}[D \xrightarrow{-} \mathcal{F}], \mathcal{F}, D)\\
\{(D_i, U_i, \mathcal{F}_i)\} &\longmapsto \{(D_i, U_i, \phi(\mathcal{F}_i))\}
\end{align*}
is constructed, such that

\begin{enumerate}
    \item[\text{(a)}] $\varphi^{\mathrm{md}}$ is compatible with pullbacks, and
    \item[\text{(b)}] if $D = \emptyset$ and we identify $S^{\mathrm{md}}(E, F, D)$ and $S^{\mathrm{md}}(E[D \xrightarrow{-} F], F, D)$, then $\varphi^{\mathrm{md}}$ becomes the identity map.
\end{enumerate}
This gives a mechanism to make sense of performing a sequence of modifications to $\cE$ one at a time, namely by canonically reinterpreting subsequent modifications of $\cE$ as modifications of an intermediate modification.

\begin{definition}{\cite[Definition 2.18]{aly}}
    Let $M$ be a tree-like modification datum for $\mathcal{E}$. If $M$ is empty, then we define $\mathcal{E}[\emptyset] = \mathcal{E}$. Otherwise, if $M = \{(D, U, \mathcal{F})\} \cup M'$, then define
    \[
        \mathcal{E}[M] = \mathcal{E}[D \xrightarrow{-} \mathcal{F}][\varphi^{\mathrm{md}}(M')].
    \]
    When $M = \{(D_1, U_1, \mathcal{F}_1), \dots, (D_m, U_m, \mathcal{F}_m)\}$, we will write
    \[
        \mathcal{E}[M] = E[D_1 \xrightarrow{-} \mathcal{F}_1] \cdots [D_m \xrightarrow{-} \mathcal{F}_m].
    \]
\end{definition}

We summarize below some properties that multi-modifications satisfy. 

\begin{proposition} \label{prop:multi_mod_properties}
Multi-modifications satisfy the following properties:
\begin{enumerate}[label=(\alph*), ref=\theproposition(\alph*)]
    \item \label{prop:pullback} \textbf{\textup{Compatibility with Pullbacks:}} Vector bundle modifications commute with pullbacks. For a morphism $f: Y \to X$ of schemes, if the support of the divisors in $M$ does not contain any component of the image of $f$, then the pullback datum $f^*M = \{(f^*D_i, f^{-1}(U_i), f^*\mathcal{F}_i)\}$ remains tree-like and fits into a natural isomorphism $f^*\mathcal{E}[M] \cong (f^*\mathcal{E})[f^*M]$ \cite[Corollary 2.19]{aly}.
    
    \item \label{prop:commutativity} \textbf{\textup{Commutativity of Modifications:}} If $M'$ is any reordering of a tree-like modification datum $M$, then there exists a natural pullback-compatible isomorphism $\mathcal{E}[M] \cong \mathcal{E}[M']$ \cite[Proposition 2.20]{aly}.
    
    \item \label{prop:twisting} \textbf{\textup{Compatibility with Twisting:}} Given a Cartier divisor $D$ on $X$, if we twist the datum to form $M(D) = \{(D_i, U_i, \mathcal{F}_i(D))\}$, then $M(D)$ is a valid tree-like datum for the twisted bundle $\mathcal{E}(D)$, yielding a natural pullback-compatible isomorphism $\mathcal{E}[M](D) \cong \mathcal{E}(D)[M(D)]$ \cite[Proposition 2.21]{aly}.

    \item \label{prop:combining} \textbf{\textup{Combining Modifications:}} Consider a tree-like modification datum 
    \[ M = \{(aD, U, \mathcal{F}_1), (bD, U, \mathcal{F}_2)\} \] 
    for $\mathcal{E}$, where $a, b$ are non-negative integers.
    \begin{enumerate}[label=(\roman*), ref=\theenumi(\roman*)]
        \item \label{prop:combining_linear} If $\mathcal{F} = \mathcal{F}_1 = \mathcal{F}_2$, then
        \[ \mathcal{E}[aD \xrightarrow{-} \mathcal{F}][bD \xrightarrow{-} \mathcal{F}] \cong \mathcal{E}[(a + b)D \xrightarrow{-} \mathcal{F}]. \]
        \item \label{prop:combining_independent} If $\mathcal{F}_1, \mathcal{F}_2$ are linearly independent and $a = b = 1$, then
        \[ \mathcal{E}[D \xrightarrow{-} \mathcal{F}_1][D \xrightarrow{-} \mathcal{F}_2] \cong \mathcal{E}[D \xrightarrow{-} \mathcal{F}_1 + \mathcal{F}_2](-D). \]
    \end{enumerate}
    In addition, both isomorphisms are compatible with pullbacks. See \cite[Proposition 2.23]{aly}.
\end{enumerate}
\end{proposition}

\begin{notation}
When it is clear that $M$ is a tree-like modification datum for $\mathcal{E}$, we will simply write $M$ in place of $M(D)$ in the situation of Proposition \ref{prop:multi_mod_properties}(c). With this notation, 
\[
\mathcal{E}[M](D)=\mathcal{E}(D)[M].
\]
That is, modifications and twists commute.
\end{notation}

The main setting where we will use modifications of bundles is the following.

\begin{situation}\label{sit: modification_normal}
Let $f:C=C' \cup L \to X_{r,s}$ be an unramified morphism out of the union of a line $L$ (of degree $(1,0)$ or $(1,1)$) and a nodal curve $C'$ meeting at a single node $q=C'\cap L$.

Let $N'_{C'\cup L}$ be a vector bundle on $C' \cup L$. Assume that $N'_{C'\cup L}$ is equipped with an isomorphism to the normal bundle $N_{C/X_{r,s}}$ over an open set of $C' \cup L$ containing the entire line $L$, and in particular containing an open neighborhood $U\subset C$ of $q$. Let $N'_{C'}$ be the vector bundle on $C'$ obtained by gluing $N'_{C'\cup L}|_{C' \setminus q}$ along $U \setminus q$, via the given isomorphism, to the restricted normal bundle $N_{C'/X_{r,s}}|_{U \cap C'}$. 

Let $v \in L$ and $w \in T_qC'$ be points distinct from $q$.
\end{situation}

The simplest example of Situation \ref{sit: modification_normal} is when $N'_{C'\cup L}=N_{C/X_{r,s}}$, in which case $N'_{C'}=N_{C'/X_{r,s}}$. More generally, we will consider the situation in which $N'_{C'\cup L}$ is a modification of $N_{C/X_{r,s}}$ at finitely many points of $C'\setminus q$, in which case $N'_{C'}$ is the corresponding modification of $N_{C'/X_{r,s}}$.

\begin{proposition}{\cite[Corollary 3.2]{HH} or \cite[Proposition 8.3]{aly}}\label{prop:hh-isomorphism}

In Situation \ref{sit: modification_normal}, we have isomorphisms
\[
N'_{C'\cup L}|_{C'} \cong N'_{C'}[q \xrightarrow{+} v] \quad \text{and} \quad N'_{C'\cup L}|_{L} \cong N_L[q \xrightarrow{+} w].
\]
\end{proposition}

\subsection{Smoothing of maps from reducible curves}

We will be interested in proving statements about maps out of a smooth rational curve by degeneration. As such, we will need the following criterion to ensure that maps out of singular curves can be smoothed.

\begin{lemma}\label{lem:maps_smooth}
    Let $X$ be a smooth, projective variety. Let $C=C_1 \cup\ldots \cup C_t$ be the decomposition into irreducible components of a nodal curve $C$, and let $f: C \to X$ be any map to $X$. Suppose that $H^1(C,f^*T_X)=0$. Then, $f$ deforms to a map from a smooth curve. More precisely, there is a flat family of stable maps $F:\mathcal{C} \to X$, where $\mathcal{C}/\Delta$ is a family of curves over the spectrum $\Delta=\mathrm{Spec}(R)$ of a DVR, with special fiber $f:C \to X$.
\end{lemma}

\begin{proof}

    Let $\mathrm{Def}(C,f)$ (resp. $\mathrm{Def}(C)$) denote the space of deformations of the map (resp. of the curve), and let $\mathrm{Ob}(f)$ denote the obstruction space for deformations of $f$ that keep the curve $C$ constant. There is an exact sequence
    \[
    \mathrm{Def}(C,f) \to \mathrm{Def}(C) \to \mathrm{Ob}(f) \cong H^1(C,f^*T_X)=0.
    \]
    Therefore, a deformation of $C$ to a smooth curve lifts to a deformation of $(C,f)$.
\end{proof}

\begin{remark}\label{rmk:H1_vanishing}
Let $\nu \colon \widetilde{C} \to C$ be the normalization map. Then, there is an exact sequence
\[
0 \to f^*T_X \to \nu_*\nu^*(f^*T_X)
   \to \bigoplus_{p \,\text{node}} T_X|_{f(p)}
   \to 0.
\]
Taking the associated long exact sequence in cohomology, we see that the condition
$
H^1(C,f^*T_X)=0
$
is equivalent to the following two conditions:
\begin{enumerate}
    \item[$\bullet$] $H^1(C_i,f^*T_X|_{C_i})=0$ for all $i=1,\ldots,t$;
    \item[$\bullet$] the evaluation map
    \[
    \bigoplus_i H^0(C_i,f^*T_X|_{C_i})
    \longrightarrow
    \bigoplus_{p\,\text{node}} T_X|_{f(p)}
    \]
    is surjective.
\end{enumerate}
\end{remark}

\subsection{The Harder-Narasimhan Filtration and the positive subsheaf}

For any vector bundle $\mathcal{E}$ on a smooth curve $C$, denote by $\mathrm{HN}_\bullet(\mathcal{E})$ the  Harder-Narasimhan filtration of $\mathcal{E}$. We refer to~$\mathrm{HN}_1(\mathcal{E})$ as the \emph{positive subsheaf} of~$\mathcal{E}$. 

When $C=\bP^1$, the vector bundle $\mathcal{E}$ splits as $\mathcal{E} \cong \cO(a_1)^{x_1} \oplus \ldots \oplus \cO(a_t)^{x_t}$ for some $a_t <\ldots <a_1$. Then, the Harder-Narasimhan filtration 
$\mathrm{HN}_\bullet(\mathcal{E})$ of $\mathcal{E}$ is 
$$
\mathrm{HN}_i(\mathcal{E})= \oplus_{j \leq i} \cO(a_j)^{x_j}.
$$

\section{Obstructions to balancedness}\label{sec:obstructions}

In \S\ref{sec:obstructions}--\ref{sec:char2}, we will be concerned with determining the normal bundle \(N_{d,k}\) of a general map \(f :C=\bP^1 \to X_{r,s}\). In this section, we explain three ways in which \(N_{d,k}\) can fail to be balanced. We will see later that these comprise a complete set of obstructions to balancedness.

\begin{notation}\label{notation:C_P1}
    For the remainder of the paper, $C$ always denotes a connected nodal curve of arithmetic genus 0. When we speak of a ``general map (curve) $f :C \to X_{r,s}$,'' it is implicit that $C=\bP^1$, unless $C$ has already been specified to be reducible.
\end{notation}

\begin{notation}\label{slope}
    Let $\cE$ be a vector bundle on $C$. Then, $\mu(\cE)=\frac{\deg(\cE)}{\mathrm{rank}(\cE)}$ denotes its slope.
\end{notation}

\subsection{Degenerate curves}\label{sec:degenerate}

Fix $r$, $s$, and suppose that $(d,k)$ is degenerate (Definition \ref{def:degenerate}). This can happen in three ways.

\begin{situation}\label{sit:degenerate}
\,
\begin{enumerate}
\item $d-k < r-s-1$ and $k \geq s+1$,
\item $d-k < r-s-1$ and $k < s+1$, and
\item $d-k \geq r-s-1$, in which case $d < r$.
\end{enumerate}
\end{situation}

Let $f \colon C \to X_{r,s}$ be a general curve of degree $(d,k)$. In each of the three cases of Situation \ref{sit:degenerate}, the image of $C$ is contained in a linear space $Y \subset X_{r,s}$, itself isomorphic to a blow-up of a lower-dimensional projective space along a linear subspace. We compute this linear space $Y$ in each of the three cases.

\begin{enumerate}
\item[(1)] Suppose first that $d-k < r-s-1$. Then, the image of $C$ in $\bP^{r-s-1}$ spans a linear subspace $Z \subset \bP^{r-s-1}$ of dimension $d-k$, cut out by $(r-s-1)-(d-k)$ general linear equations. The pullback $\pi^{-1}(Z) \subset X_{r,s}$ is isomorphic to $\Bl_{\bP^s}(\bP^{s+1+d-k})$. If $k \geq s+1$, then $C$ is a general, non-degenerate curve of degree $(d,k)$ on $Y := \pi^{-1}(Z)$.

\item[(2)] If $k < s+1$, then the degree $d$ of $C$ in $\bP^{s+1+d-k}$ is strictly less than $\dim \Bl_{\bP^s}(\bP^{s+1+d-k})$, so $C$ lies on a further linear subspace $Y \subset \pi^{-1}(Z)$. The subspace $Y$ is cut out by $s+1-k$ additional general linear equations, pulled back from $\bP^r$. Then, $Y \cong \Bl_{\bP^{k-1}}(\bP^d)$, and $C$ is a general, non-degenerate curve of degree $(d,k)$ on $Y$. When $k=0$, we set $\Bl_{\bP^{k-1}}(\bP^d)=\bP^d$. 

\item[(3)] Finally, if $d-k \geq r-s-1$, then the image of $C$ in $\bP^{r-s-1}$ is non-degenerate, but its span $Y$ in $X_{r,s}$ is cut out by $r-d$ general linear equations pulled back from $\bP^r$. Then, $Y \cong \Bl_{\bP^{s-(r-d)}}(\bP^d)$, and $C$ is once more a general, non-degenerate curve on $Y$.
\end{enumerate}

In all three cases, we have a short exact sequence
\[
0 \to T_Y \to T_{X_{r,s}}|_Y \to N_{Y/X_{r,s}} \to 0.
\]
Let $Y' \subset X_{r,s}$ be a linear subspace of dimension $\dim X_{r,s}-\dim Y-1$, with the property on the one hand that $Y \cap Y' = \emptyset$, but on the other hand that the images of $Y,Y'$ span both $\bP^r$ and $\bP^{r-s-1}$ under the respective projections. Then, the rational map $X_{r,s} \dashrightarrow Y$ given by projecting from $Y'$ induces a splitting $T_{X_{r,s}}|_Y \to T_Y$.

In particular, after passing to the quotient by $T_C$, we have $N_{(d,k;r,s)} \cong \left. N_{Y/X_{r,s}} \right|_C \oplus N_{C/Y}$ in all three cases. Because $Y \subset X_{r,s}$ is cut out by hyperplane sections in all three cases, we may compute the following.
\begin{proposition}\label{prop:degenerate_splittings}
In the three cases of Situation \ref{sit:degenerate}, we have
\begin{enumerate}
\item If $d-k < r-s-1$ and $k \geq s+1$, then
\[
N_{(d,k;r,s)} \cong N_{(d,k;s+1+d-k,s)} \oplus \cO(d-k)^{(r-s-1)-(d-k)}.
\]
\item If $d-k < r-s-1$ and $k < s+1$, then
\[
N_{(d,k;r,s)} \cong N_{(d,k;d,k-1)} \oplus \cO(d-k)^{(r-s-1)-(d-k)} \oplus \cO(d)^{s+1-k}.
\]
where when $k=0$, we set $N_{(d,k;d,k-1)}= N_{C/\bP^d}$.
\item If $d-k \geq r-s-1$ and $d < r$, then
\[
N_{(d,k;r,s)} \cong N_{(d,k;d,s-(r-d))} \oplus \cO(d)^{r-d}.
\]
\end{enumerate}
\end{proposition}

In particular, the normal bundles of general, degenerate curves are determined by those of general, non-degenerate curves.

\begin{corollary}\label{cor:most_degenerate_not_balanced}
Suppose that $d-k < r-s-1$ (cases (1) and (2) of Situation \ref{sit:degenerate}). Then, $N_{(d,k;r,s)}$ is not balanced, unless $(d,k)=(1,0)$.
\end{corollary}

\begin{proof}
Consider first case (1) of Proposition \ref{prop:degenerate_splittings}. We claim that $\mu\left(N_{(d,k;s+1+d-k,s)}\right) > d-k+1$, which implies that $N_{(d,k;r,s)}$ cannot be balanced. This inequality is equivalent to
\[
(s+2+d-k)d-(d-k)k-2 > (d+s-k)(d-k+1),
\]
which in turn is equivalent to $d+k-2+s(k-1)>0$. This last inequality holds because $d > k \geq s+1 \geq 1$.

Next, consider case (2) of Proposition \ref{prop:degenerate_splittings}. Clearly, $N_{(d,k;r,s)}$ can only be balanced if $k=0,1$. If $k=0$, then by Theorem~\ref{thm:Pr}, $N_{(d,k;r,s)}$ cannot be balanced unless $d=1$. If $k=1$, then $\mu\left(N_{(d,1;d,0)}\right)=d+1$, so $N_{(d,1;r,s)}$ cannot be balanced.
\end{proof}

\subsection{Projection sequence obstructions}\label{sec:splitting}
For $s < r-2$, recall the exact sequence \eqref{eq:proj_seq} 
\begin{equation*}
0\to T_{\rel} \to N_{d,k}\to N_{C/\bP^{r-s-1}} \to 0.
\end{equation*}

We have:
\begin{itemize}
\item $\deg(T_{\rel})=(s+1)d+k$ and $\rk(T_{\rel})=s+1$,
\item $\deg(N_{d,k})=(r+1)d-(r-s-1)k-2$ and $\rk(N_{d,k})=r-1$,
\item $\deg(N_{C/\bP^{r-s-1}})=(r-s)(d-k)-2$ and $\rk(N_{C/\bP^{r-s-1}})=r-s-2$.
\end{itemize}

\begin{lemma}\label{lem:Trel_balanced}
$T_{\rel}$ is balanced.
\end{lemma}
\begin{proof}
We first compute $T_{\rel}$ in the cases $(d,k) = (1,0)$ and for $d=k$. While we typically assume $d>k$ throughout this paper, it is convenient in this proof to allow $d=k$ for the inductive argument.

Suppose $(d,k) = (1,0)$. Identifying $X_{r,s} \cong \bP(\cO_{\bP^{r-s-1}}(-1) \oplus \cO^{s+1})$, we obtain the relative Euler sequence
\[
0 \to \cO \to \cO_{X_{r,s}}(1) \otimes \bigg( \cO_{\bP^{r-s-1}}(-1) \oplus \cO^{s+1} \bigg) \to T_{X_{r,s}/\bP^{r-s-1}} \to 0
\]
on $X_{r,s}$. Restricting to $C$, this becomes
\[
0 \to \cO \xrightarrow{f} \cO \oplus \cO(1)^{ s+1} \to T_{\rel} \to 0.
\]
The injection $f$ is determined by its component maps
\[
f = (c, s_1, s_2, \dots, s_{s+1}),
\]
where $c \in H^0(\mathcal{O}) \cong \mathbb{C}$ and $s_i \in H^0(\mathcal{O}(1))$. The component $c$ is the restriction to $C$ of the canonical section of $\mathcal{O}_{X_{r,s}}(1) \otimes \mathcal{O}_{\bP^{r-s-1}}(-1) \cong \mathcal{O}_{X_{r,s}}(E)$ defining $\mathcal{E}$.

Since $C$ does not meet the blow-up center $\mathbb{P}^s$, the section $c$ is nowhere-vanishing along $C$. In particular, $c \neq 0$, so $f$ projects isomorphically onto the first summand $\mathcal{O}$ of the middle term. It follows that
\[
T_{\mathrm{rel}}|_{C} \simeq \frac{\mathcal{O} \oplus \mathcal{O}(1)^{ s+1}}{\mathcal{O}} \simeq \mathcal{O}(1)^{ s+1}.
\]

Now, suppose $d=k$. Then, $C$ lies in a fiber $\bP^{s+1}$ of the projection $X_{r,s} \to \bP^{r-s-1}$, and $T_{\rel}|_{C}= T_{\bP^{s+1}}|_{C}$ is balanced by Theorem \ref{thm:Pr_tangent}.

Finally, suppose $d > k \geq 0$. We proceed by induction on $d$. Let $C'\cup L$ be a union of smooth curves meeting at a single node $q$, and by slight abuse of notation, let $f:C'\cup L\to X_{r,s}$ be a map such that the restrictions of $f$ to $C'$ and $L$ have degrees $(d-1, k)$ and $(1,0)$, respectively, and are general. 

By Lemma \ref{lem:maps_smooth} and Remark \ref{rmk:H1_vanishing}, the map $f$ smooths to a map of degree $(d,k)$ with smooth domain. Indeed, Lemma \ref{lem:free} shows that the hypothesis on vanishing $H^1$ is satisfied, and the evaluation map is surjective, since
\[
f^*T_X|_{L} \simeq \cO(1)^{\oplus (r-2)} \oplus \cO(2)^{\oplus 2}.
\] Because interpolation is open in flat families, it suffices to show that the restriction of the relative tangent bundle of $f:C'\cup L\to X_{r,s}$ to $C'\cup L$ satisfies interpolation.

Let $D$ be the divisor on $L$ consisting of a reduced point in $L \smallsetminus \{q\}$.
Observe that
\[
H^0( L, T_{\rel}(-q-D)) = 0.
\]
As in Lemma \ref{lem: criterion interpolation}, define
\[
V = \mathrm{ev}_{C'}^{-1}\big(\mathrm{ev}_L(H^0(L, T_{\rel}|_L(-D)))\big).
\]

Then, $V = H^0(C', T_{\rel}|_{C'})$, which satisfies interpolation by the inductive hypothesis, and has dimension
\[
\chi(T_{\rel}|_{C'}) = \chi(T_{\rel}|_{C'}) + \chi\big(T_{\rel}|_L(-D - q)\big).
\]
The result follows from Lemma \ref{lem: criterion interpolation}.

\end{proof}

\begin{proposition}\label{prop:projection_split_char_not2}
Suppose that $s \neq r-2$, and that either
\begin{enumerate}
\item[(i)] $\charac(\bK)\neq 2$ and $d-k \geq r-s-1$, or 
\item[(ii)] $\charac(\bK)=2$ and $(r-s-2)\mid (d-k-1)$.
\end{enumerate}
Then, we have the following.
\begin{enumerate}
\item[(a)] The exact sequence \eqref{eq:proj_seq} splits if
\[
\mu(N_{C/\bP^{r-s-1}})=\frac{(r-s)(d-k)-2}{r-s-2} \leq \left\lfloor \frac{(s+1)d+k}{s+1} \right\rfloor + 1 = \left\lfloor \mu(T_{\rel}) \right\rfloor + 1.
\]
This inequality is equivalent to
\[
d \leq \frac{(r-s)(k+1)+(r-s-2)\left\lfloor \frac{k}{s+1} \right\rfloor}{2} =: U_{r,s}(k).
\]
\item[(b)] If $d \leq U_{r,s}(k)$, then $N_{d,k} \cong T_{\rel} \oplus N_{C/\bP^{r-s-1}}$ is balanced if and only if
\[
\mu(N_{C/\bP^{r-s-1}})=\frac{(r-s)(d-k)-2}{r-s-2} \geq \left\lceil \frac{(s+1)d+k}{s+1} \right\rceil - 1 = \left\lceil \mu(T_{\rel}) \right\rceil - 1.
\]
This inequality is equivalent to
\[
d \geq \frac{(r-s)(k-1)+4+(r-s-2)\left\lceil \frac{k}{s+1} \right\rceil}{2} =: L_{r,s}(k).
\]
\end{enumerate}
In particular, if $d<L_{r,s}(k)$, then $N_{d,k}$ is not balanced.
\end{proposition}

\begin{proof}
    The map $f:C \longrightarrow X_{r,s}$ is given by 
    \[
    [f_0f_1:\cdots:f_0f_{r-s}:f_{r-s+1}:\cdots:f_{r+1}],
    \]
    where $f_{r-s+1},\ldots,f_{r+1}\in H^0(C,\cL)$ are sections of a line bundle $\cL$ of degree $d$, $f_0\in H^0(C,\mathcal{O}(D))$ for a divisor $D$ of degree $k$, and $f_1,\ldots,f_{r-s}\in H^0(C,\cL(-D))$. The induced map to \(\mathbb{P}^{r-s-1}\) is determined by the sections $f_1,\ldots,f_{r-s}$, and hence its image is a general rational curve of degree $d-k$. In particular, the normal bundle $N_{C/\mathbb{P}^{r-s-1}}$ is balanced by Theorem \ref{thm:Pr}. The restriction of $T_{\rel}$ is also balanced by Lemma \ref{lem:Trel_balanced}.

    Thus, the inequality $d \leq U_{r,s}(k)$ is precisely the condition ensuring that
    \[
    \operatorname{Ext}^1\!\bigl(
    N_{C/\mathbb{P}^{r-s-1}},
    \, T_{\mathrm{rel}}|_{C}
    \bigr)=0,
    \]
    while the range $L_{r,s}(k)\leq d\leq U_{r,s}(k)$ is exactly the range for which the resulting split bundle is balanced.
\end{proof}

\begin{definition}\label{def:UL_edge}
By convention, we define $U_{r,r-2}(k)=L_{r,r-2}(k)=k+1$, consistent with the formulas of Proposition \ref{prop:projection_split_char_not2}.
\end{definition}

\begin{situation}\label{sit:char2_notdiv}
Suppose that $\charac(\bK) = 2$, that $s \neq r-2$, and that $(r-s-2) \nmid (d-k-1)$. Let $\epsilon \in [1,r-s-3]$ be the residue of $d-k-1 \pmod{r-s-2}$, so that
\[
N_{C/\bP^{r-s-1}} \cong \cO(a)^{r-s-2-\epsilon} \oplus \cO(a+2)^\epsilon,
\]
where
\[
a = (d-k)+2 \cdot \left\lfloor \frac{d-k-1}{r-s-2} \right\rfloor.
\]
\end{situation}

\begin{proposition}\label{prop:projection_split_char2}
In Situation \ref{sit:char2_notdiv}, if
\[
\mu(T_{\rel}) \geq a+1
\qquad \Longleftrightarrow \qquad
\frac{s+2}{s+1} \cdot k - 1 \geq 2 \cdot \left\lfloor \frac{d-k-1}{r-s-2} \right\rfloor,
\]
then the exact sequence \eqref{eq:proj_seq} splits, and $N_{d,k} \cong T_{\rel} \oplus N_{C/\bP^{r-s-1}}$ is not balanced.
\end{proposition}

\begin{proof}
    Similar to Proposition \ref{prop:projection_split_char_not2}.
\end{proof}

\begin{proposition}\label{prop:projection_char2_low_obstruction}
In Situation \ref{sit:char2_notdiv}, if
\[
\deg(N_{d,k}) > (r-s-2-\epsilon)a + (s+1+\epsilon)(a+1)
\Longleftrightarrow
d+(s+1)k-s-2 > (r+s)\left\lfloor \frac{d-k-1}{r-s-2} \right\rfloor,
\]
then $N_{d,k}$ is not balanced.
\end{proposition}

\begin{proof}
Suppose for sake of contradiction that $N_{d,k}$ is balanced. If $\mu(N_{d,k}) \geq a+1$, then $N_{d,k}$ cannot surject onto
\[
N_{C/\bP^{r-s-1}} \cong \cO(a)^{r-s-2-\epsilon} \oplus \cO(a+2)^\epsilon.
\]
Otherwise,
\[
N_{d,k} \cong \cO(a)^x \oplus \cO(a+1)^y,
\]
where $x < r-s-2-\epsilon$. Again, $N_{d,k}$ cannot surject onto $N_{C/\bP^{r-s-1}}$.
\end{proof}

\subsection{Propagation of unbalancedness of $(d,0)$}

\begin{lemma}\label{lem:char2_obstruction_general}
Suppose $\charac(\bK) = 2$. With notation as in Theorem \ref{thm:Pr}(b), we have
\begin{enumerate}
\item[(a)] $h^0(N_{d,k}(-\alpha-2))\ge \ell-k(r-s-1)$, and
\item[(b)] $h^1(N_{d,k}(-\alpha-2+k)) \ge (r-1-\ell)-ks$.
\end{enumerate}
\end{lemma}

\begin{proof}
Let $f:C\to X_{r,s}$ be a general curve of degree $(d,k)$, which may also be viewed as a curve in $\bP^r$ via the blow-up $b:X_{r,s}\to \bP^r$. We have short exact sequences
\[
0\to N_{d,k}(-\alpha-2)\to N_{C/\bP^r}(-\alpha-2) \to \bigoplus_{p\in C\cap \bP^s} (N_{C/\bP^r}|_p/T_{\bP^s}|_p) \to 0
\]
and (see \cite[Equation (2.2)]{aly})
\[
0\to N_{C/\bP^r}(-\alpha-2) \to N_{d,k}((-\alpha-2)+k) \to \bigoplus_{p\in C\cap \bP^s} T_{\bP^s}|_p \to 0.
\]
The composition $b \circ f:C\to \bP^r$ may not be a general rational curve of degree $d$ in $\bP^r$, but $N_{C/\bP^r}$ is certainly a specialization of $N_{d,0}$, so by semi-continuity and Theorem \ref{thm:Pr}(b), we have $h^0(N_{C/\bP^r}(-\alpha-2))\ge \ell$ and $h^1(N_{C/\bP^r}(-\alpha-2))\ge r-1-\ell$. The conclusions follow from the long exact sequence of cohomology groups. 
\end{proof}

\begin{corollary}\label{cor:char2_obstruction}
Suppose $\charac(\bK) = 2$. Fix $r,s,d,k$. Write $\ell \in [0,r-2]$ for the residue of $d-1 \pmod{r-1}$, and suppose that $\ell > 0$. If
\begin{equation}\label{eq:char2_obstruction_bound}
k < \max\left( \left\lceil \frac{\ell}{r-s-1} \right\rceil, \left\lceil \frac{r-1-\ell}{s} \right\rceil \right),
\end{equation}
then $N_{d,k}$ is not balanced.

If $s = 0$, then the right-hand side of \eqref{eq:char2_obstruction_bound} is taken to be $\infty$. In particular, if $s=0$ and $N_{d,k}$ is balanced, then $(r-1) \mid (d-1)$.
\end{corollary}

\begin{proof}
Assume that $k< \left\lceil \frac{\ell}{r-s-1} \right\rceil$, or equivalently that $\ell-k(r-s-1)\ge 1$. By Lemma \ref{lem:char2_obstruction_general}(a), $N_{d,k}$ contains summands of degree at least $\alpha+2$. If $N_{d,k}$ is balanced, then its degree can furthermore be at least $(\alpha+1)(r-1)+\ell-k(r-s-1)$. On the other hand, we have
\[
\deg(N_{d,k})=(r-1)\alpha+2\ell-(r-s-1)k,
\]
and the inequality $(r-1)\alpha+2\ell-(r-s-1)k \ge (\alpha+1)(r-1)+\ell-k(r-s-1)$ is equivalent to $\ell\ge r-1$, a contradiction.

If instead $k< \left\lceil \frac{r-1-\ell}{s} \right\rceil $, or equivalently (valid also when $s=0$), if $(r-1-\ell)-ks\ge 1$, then Lemma \ref{lem:char2_obstruction_general}(b) shows similarly that $\deg(N_{d,k})\le (\alpha-k)(r-1)+(\ell+ks)$. This simplifies to $\ell\le 0$, contradicting the assumption that $\ell>0$.
\end{proof}

\section{Inductive steps}\label{sec:inductive}

\subsection{Main arguments}

Our main inductive arguments, Propositions \ref{prop:main_inductive} and \ref{prop:main_inductive_refinement}, compute $N_{d+(r-s-1),k+1}$ from $N_{d,k}$, assuming good properties of $N_{d,k}$.

\begin{definition}
Fix $r,s$, and assume in addition that $s>0$. We say that $N_{d,k}$ is \emph{almost balanced} if either it is balanced, or $N_{d,k}\cong \cO(a)^x\oplus \cO(a+1)^y\oplus \cO(a+2)^z$ for some integers $a,x,y,z$ with $x+y+z=r-1$, and in addition,
\begin{equation*}
\frac{x}{r-s-1},\frac{z}{s}\le 1.
\end{equation*}

More generally, for any positive integer $m$, we say that  $N_{d,k}$ is \emph{$m$-almost balanced} if either it is balanced, or, with notation as above,
\begin{equation*}
\frac{x}{r-s-1},\frac{z}{s}\le m.
\end{equation*}
\end{definition}
By convention, we consider ``$0$-almost balanced'' to be synonymous with ``balanced.''

\begin{proposition}\label{prop:main_inductive}
Fix $r,s$. (We allow $s=0$ in (a).)
\begin{enumerate}
\item[(a)] If $N_{d,k}$ is balanced, then so is $N_{d+(r-s-1),k+1}$. 
\item[(b)] If $N_{d,k}$ is $m$-almost balanced, then $N_{d+(r-s-1),k+1}$ is $(m-1)$-almost balanced. In particular, if $N_{d,k}$ is $m$-almost balanced, then $N_{d+m(r-s-1),k+m}$ is balanced.
\end{enumerate}
\end{proposition}

\begin{proposition}\label{prop:main_inductive_refinement}
Fix $r,s$. Suppose that $N_{d,k}\cong \cO(a)^x\oplus \cO(a+1)^y\oplus \cO(a+2)^z$.
\begin{enumerate}
\item[(a)] Suppose that $x\ge r-s-1$. Then, 
\[
N_{d+(r-s-1),k+1}\cong \cO(a+(r-s-1))^{x'}\oplus \cO((a+1)+(r-s-1))^{y'}\oplus \cO((a+2)+(r-s-1))^{z'},
\]
where $x'\le x-(r-s-1)$ and $z'\le z$.
\item[(b)] Suppose that $z\ge s$. Then,
\[
N_{d+(r-s-1),k+1}\cong \cO(a+(r-s))^{x'}\oplus \cO((a+1)+(r-s))^{y'}\oplus \cO((a+2)+(r-s))^{z'},
\]
where $x'\le x$ and $z'\le z-s$.
\end{enumerate}
\end{proposition}

The proofs of Propositions~\ref{prop:main_inductive} and~\ref{prop:main_inductive_refinement} will both use the following construction.

\begin{construction}\label{construction}
Let  
\[
C= C' \cup \bigg( \bigcup_{i=1}^{r-s-2} L_i \bigg) \cup M
\]
be the union of a smooth rational curve $C'$ attached to smooth rational components $L_1,\ldots,L_{r-s-2},M$ at distinct nodes $q_1,\ldots,q_{r-s-2},q$, respectively.

Let $f: C \to X_{r,s}$ be a general map, with the property that the restriction of $f$ to $C'$ has degree $(d,k)$, the restriction of $f$ to each $L_i$ has degree $(1,0)$, and the restriction of $f$ to $M$ has degree $(1,1)$. Assume that $f(q)$ does not lie in the exceptional divisor of $X_{r,s}$. Let $Y \subseteq X_{r,s}$ be the fiber of $\pi: X_{r,s} \to \bP^{r-s-1}$ containing $M$.

\begin{center}
\begin{tikzpicture}[scale=1.3, every node/.style={font=\small, text=black}]

    \coordinate (x1) at (0.5, 0.65);
    \coordinate (x2) at (1.8, 1.1);
    \coordinate (xr) at (4.0, 0.7);
    \coordinate (p)  at (6.0, 0.2);

    \draw[thick] plot[smooth, tension=0.6] 
        coordinates {(-0.8, 0.1) (x1) (x2) (3.0, 1.0) (xr) (5.0, 0.5) (p) (7.2, -0.2)};
    \node[right] at (7.2, -0.2) {$C'$};

    \draw[thick] ($(x1) + (265:1.2)$) -- ($(x1) + (85:1.6)$) node[above] {$L_1$};
    \fill[black] ($(x1) + (85:1.0)$) circle (1.8pt) node[right] {$v_1$};

    \draw[thick] ($(x2) + (270:1.2)$) -- ($(x2) + (90:1.6)$) node[above] {$L_2$};
    \fill[black] ($(x2) + (90:1.0)$) circle (1.8pt) node[right] {$v_2$};

    \node at (2.9, 0.6) {$\cdots$};

    \draw[thick] ($(xr) + (260:1.2)$) -- ($(xr) + (80:1.6)$) node[above] {$L_{r-s-2}$};
    \fill[black] ($(xr) + (80:1.0)$) circle (1.8pt) node[right] {$v_{r-s-2}$};

    \draw[thick] ($(p) + (230:1.2)$) -- ($(p) + (50:1.6)$) node[above right] {$M$};
    \fill[black] ($(p) + (50:1.0)$) circle (1.8pt) node[above left] {$v$};

    \fill[black] (x1) circle (1.8pt) node[below right, xshift=-2pt] {$q_1$};
    \fill[black] (x2) circle (1.8pt) node[below right, xshift=-2pt] {$q_2$};
    \fill[black] (xr) circle (1.8pt) node[below right, xshift=-2pt] {$q_{r-s-2}$};
    \fill[black] (p) circle (1.8pt) node[below left, xshift=6.5pt] {$q$};

\end{tikzpicture}
\end{center}
\end{construction}

By Lemma \ref{lem:maps_smooth} and Remark \ref{rmk:H1_vanishing}, a map $f$ as in Construction \ref{construction} smooths to a map of degree $(d+(r-s-1),k+1)$ out of a smooth curve. We will deduce properties of $N_{d+(r-s-1),k+1}$ from spaces of sections of the normal bundle $N_{C/X_{r,s}}$ of the singular curve in Construction \ref{construction}, twisted by appropriate divisors.

\begin{lemma}\label{lem:construction_restricted_normal}
In the setting of Construction \ref{construction}, we have the following.
\begin{enumerate}[label=(\alph*), ref=\theproposition(\alph*)]
\item \label{lemma:modified_normal_dk}
The restriction of $N_{C/X_{r,s}}$ to $C'$ is given by
\[
N_{C/X_{r,s}}|_{C'} = N_{C'/X_{r,s}}[q_1 \xrightarrow{+} v_1] \cdots [q_{r-s-2} \xrightarrow{+} v_{r-s-2}][q \xrightarrow{+} v],
\]
\item \label{lemma:modified_normal_line}For all $i=1,\ldots,r-s-2$, the restriction of $N_{C/X_{r,s}}$ to $L_i$ is given by
\[
N_{C/X_{r,s}}|_{L_i} = N_{L_i/X_{r,s}}[q_i \xrightarrow{+} T_{q_i}C'] \cong \cO_{L_i}(1)^{ r-s-2} \oplus \cO_{L_i}(2)
\]
\item The restriction of $N_{C/X_{r,s}}$ to $M$ is given by
\[
N_{C/X_{r,s}}|_{M} = N_{M/X_{r,s}}[q \xrightarrow{+} T_qC']
\cong  \cO^{ r-s-2} \oplus \cO_{M}(1)^{ s+1}.
\]
\end{enumerate}
\end{lemma}

\begin{proof}
    To prove (a), apply Proposition \ref{prop:hh-isomorphism}, to the components $L=L_1,\ldots,L_{r-s-2},M$, in order. At the $j$-th step, the bundle $N'_{C'\cup L}$ is taken to be
    \[
    N_{C'/X_{r,s}}[q_1 \xrightarrow{+} v_1] \cdots [q_{j-1} \xrightarrow{+} v_{j-1}].
    \]
    To prove (b) and (c), apply Proposition \ref{prop:hh-isomorphism} with $N'_{C'\cup L}=N_{C/X_{r,s}}$.
\end{proof}

We next match the fiber of the positive subsheaf of each $N_{C/X_{r,s}}|_{L_i}$ (resp. $N_{C/X_{r,s}}|_{M}$) over each nodal point $q_i$ (resp. $q$) with the fibers of natural subsheaves of $N_{C/X_{r,s}}|_{C'}$.

\begin{lemma}\label{lemma:positive_subsheaves_at_q}
\,
\begin{enumerate}
\item[(a)] Consider the subsheaf 
\[
N_{C' \to v_i}(q_i)\subset N_{C/X_{r,s}}|_{C'}.
\]
Then, the fiber of $N_{C' \to v_i}(q_i)$ at $q_i\in C$ coincides with the fiber of the positive subsheaf $\cO_{L_i}(2)\subset N_{C/X_{r,s}}|_{L_i}$ (see Lemma \ref{lem:construction_restricted_normal}(b)) at $q_i$.
\item[(b)] Consider the subsheaf 
\[
T_{X_{r,s}/\bP^{r-s-1}}|_{C'}[q \xrightarrow{+} v]\subset N_{C/X_{r,s}}|_{C'}.
\]
Then, the fiber of $T_{X_{r,s}/\bP^{r-s-1}}|_{C'} [q \xrightarrow{+} v]$ at $q\in C$ coincides with the fiber of the positive subsheaf $\cO_{M}(1)^{s+1}\subset N_{C/X_{r,s}}|_{M}$ (see Lemma \ref{lem:construction_restricted_normal}(c)) at $q$.
\end{enumerate}
\end{lemma}

\begin{proof}
We start with (a). For $i=1,\ldots,r-s-2$, let $w_i \in T_{q_i}C' \smallsetminus \{q_i\}$ be any point. Recall that we have an inclusion of bundles $N_{L_i \to w_i} \subseteq N_{L_i/X_{r,s}}$ from \S\ref{sec: pointing bundles}. By Lemma \ref{lemma:modified_normal_line}, the restriction of $N_{C/X_{r,s}}$ to $L_i$ is precisely the modified bundle
$$
N_{L_i}[q_i \xrightarrow{+} w_i]=N_{L_i}[q_i \xrightarrow{+} N_{L_i \to w_i}].
$$
Thus, we have a natural inclusion $\cO_{L_i}(2) \cong N_{L_i \to w_i}(q_i) \subseteq N_{C}|_{L_i}$ identifying the positive subsheaf of 
$ N_{C/X_{r,s}}|_{L_i}
$
with
$
N_{L_i \to w_i}(q_i). 
$
Moreover, by \cite[Proposition 8.4]{aly}, one has
\begin{equation*}
N_{L_i \to w_i}(q_i)|_{q_i}
=
N_{C' \to v_i}(q_i)|_{q_i}
\subseteq
N_{C/X_{r,s}}|_{q_i},
\qquad i=1,\ldots,r-s-2.
\end{equation*}
This concludes the proof of (a). 

Next, we prove (b). Let $w \in T_{q}C' \smallsetminus \{q\}$  be any point.
Consider the subbundle $N_{M \to w} \subseteq N_{M/X_{r,s}}$. We claim that $N_{M \to w}$ does not lie in the positive subsheaf of 
\[
\cO_M(1)^{s}\subset N_{M/X_{r,s}}\cong \cO_M^{r-s-1}\oplus \cO_M(1)^{s}.
\] 
Indeed, the positive subsheaf of $N_{M/X_{r,s}}$ corresponds to the subbundle of deformations inside the fiber $Y$ of $\pi:X_{r,s}\to \bP^{r-s-1}$ containing $M$, and if $d>k$, then $T_{q}C'$ is not contained in this fiber. 

In particular, there is an inclusion of bundles
\[
\cO_M \cong N_{M \to w} \subseteq N_{M/X_{r,s}},
\]
and hence
\[
\cO_{M}(1) \cong N_{M \to w}(q) \subseteq N_{C/X_{r,s}}|_M.
\]

The natural map 
\[
N_{M/\bP^{s+1}} \oplus N_{M \to w}(q) \to N_{C/X_{r,s}}|_{M}
\]
identifies $N_{M/\bP^{s+1}} \oplus N_{M \to w}(q)$ with the positive subsheaf $\cO_{M}(1)^{s+1}$ of $N_{C/X_{r,s}}|_{M}$.

We claim that the fiber at $q$ of $\cO_{M}(1)^{s+1} = N_{M/Y} \oplus N_{M \to w}(q) \subseteq N_{C/X_{r,s}}|_M$ coincides with
\begin{equation*}
    \cO_{M}(1)^{ s+1}|_q= T_{X_{r,s}/\bP^{r-s-1}}|_{C'}[q \xrightarrow{+} v]|_q
\end{equation*}
when viewed in $N_{C/X_{r,s}}|_q$. To see this, it suffices to show that 
\[
(N_{M/Y})|_q \oplus (N_{M \to w}(q))|_q\subset T_{X_{r,s}/\bP^{r-s-1}}|_{C'}[q \xrightarrow{+} v]|_q
\]
as subspaces of $N_{C/X_{r,s}}|_q$. By \cite[Proposition 8.4]{aly}, we have
\[
N_{M \to w}(q)|_q= N_{C' \to v}(q)|_q \subseteq N_C|_{q}
\]
which lies in $T_{X_{r,s}/\bP^{r-s-1}}|_{C'}[q \xrightarrow{+} v]|_q$ because $N_{C' \to v}\subseteq T_{X_{r,s}/\bP^{r-s-1}}|_{C'}$. The inclusion $(N_{M/Y})|_q \subseteq T_{X_{r,s}/\bP^{r-s-1}}|_{C'}[q \xrightarrow{+} v]|_q$ instead follows from the commutative diagram 

\begin{equation*}
\begin{tikzcd}
T_Y|_q = T_{X_{r,s}/\bP^{r-s-1}}|_q
  \arrow[r]
  \arrow[d, two heads]
&
T_{X_{r,s}/\bP^{r-s-1}}|_{C'}[q \xrightarrow{+} v]|_q
  \arrow[d, hookrightarrow]
\\
N_{M/Y}|_q
  \arrow[r, hookrightarrow]
&
N_{C/X_{r,s}}|_q
\end{tikzcd}
.
\end{equation*}

\end{proof}

\begin{notation}\label{notation:D}
Consider the effective divisor on $C$ given by
\[
    D = 2v_1 + \cdots + 2v_{r-s-2} + v + D',
\]
where $D'$ is a divisor supported on
$
    C \setminus \bigg( \bigcup_i L_i \cup M \bigg),
$ to be specified later.
\end{notation}

\begin{lemma}\label{lem:H0_from_C_to_C'}
Let $D\subset C$ be as in Notation \ref{notation:D}. We have 
\begin{equation*}
 H^0(C, N_{C/X_{r,s}}(-D))  \cong H^0(C', N_{C/X_{r,s}}(-D)|_{C'}[q_1 \xrightarrow{-}v_1] \ldots [q_{r-s-2} \xrightarrow{-}v_{r-s-2}][q \xrightarrow{-} T_{X_{r,s}/\bP^{r-s-1}}|_{C'}]).
\end{equation*}
\end{lemma}

\begin{proof}
We have
\begin{align*}
    H^0(L_i, N_{C/X_{r,s}}(-D-q_i)|_{L_i})&=0,\\
    H^0(M,  N_{C/X_{r,s}}(-D-q)|_{M})&=0,
\end{align*}
for $i=1,\ldots, r-s-2$. Thus, the restriction yields an inclusion
\begin{equation}\label{eqn: inclusion_I}
    H^0(C, N_{C/X_{r,s}}(-D)) \hookrightarrow H^0(C', N_{C/X_{r,s}}(-D)|_{C'}).
\end{equation}
The image of the map \eqref{eqn: inclusion_I} coincides with the subspace of sections $ \sigma \in H^0(C', N_{C/X_{r,s}}(-D)|_{C'})$ such that 
\[
\begin{aligned}
(\sigma(q_1), \ldots, \sigma(q_{r-s-2}),\sigma(q)) \in \operatorname{im} \Bigg( 
    &\bigg( \bigoplus_{i=1}^{r-s-2} H^0(L_i, N_{C/X_{r,s}}(-D)|_{L_i}) \bigg) \oplus H^0(M, N_{C/X_{r,s}}(-D)|_{M}) \Bigg. \\
    &\Bigg. \xrightarrow{\mathrm{ev}} \bigg( \bigoplus_{i=1}^{r-s-2} N_{C/X_{r,s}}|_{q_i} \bigg) \oplus N_{C/X_{r,s}}|_q \Bigg)
\end{aligned}
\]
By Lemma \ref{lemma:positive_subsheaves_at_q}(a), the image of the evaluation map $\mathrm{ev}$ is equal to
\[
N_{C' \to v_1}(q_1)|_{q_1} \oplus \ldots \oplus N_{C' \to v_{r-s-2}}(q_{r-s-2})|_{q_{r-s-2}} \oplus T_{X_{r,s}/\bP^{r-s-1}}|_{C'}[q \xrightarrow{+} v]|_{q},
\]
which yields the needed identification.
\end{proof}

\begin{proposition}\label{prop:bound_collide_points}
Let $\Lambda \subset N_{C'/X_{r,s}}|_q$ be a general subspace of dimension $r-s-1$. Then, for $i=0,1$, we have
\[
h^i(C, N_{C/X_{r,s}}(-D))\le h^i(C',N_{C'/X_{r,s}}(-D'-(r-s-2)q)[q\xrightarrow{+} \Lambda]).
\]
\end{proposition}

\begin{proof}
A straightforward computation shows that 
\[
N_{C/X_{r,s}}(-D) \qquad \text{and} \qquad N_{C'/X_{r,s}}(-D'-(r-s-2)q)[q\xrightarrow{+} \Lambda]
\]
have the same degree and rank, hence the same Euler characteristic. It therefore suffices to consider the case $i=0$.

By Lemma \ref{lem:H0_from_C_to_C'}, it suffices to bound from above the dimension of the space of global sections of 
    
    \[
 N_{C/X_{r,s}}(-D)|_{C'}[q_1 \xrightarrow{-}v_1] \ldots [q_{r-s-2} \xrightarrow{-}v_{r-s-2}][q \xrightarrow{-} T_{X_{r,s}/\bP^{r-s-1}}|_{C'}].
 \]
 By Lemma \ref{lemma:modified_normal_dk} and Proposition \ref{prop:twisting}, we also have
\[
N_{C/X_{r,s}}(-D)|_{C'}\cong N_{C'/X_{r,s}}(-D')[q_1 \xrightarrow{+}v_1] \ldots [q_{r-s-2} \xrightarrow{+}v_{r-s-2}][q \xrightarrow{+} v]
\]

By semi-continuity, it suffces in turn to assume that $q_i=q$ for all $i=1,\ldots,r-s-2$. We are also free to choose the points $v_1,\ldots,v_{r-s-2},v$, as follows. Because
\[
\dim(\Lambda) + \rank(T_{X_{r,s}/\bP^{r-s-1}}) = r > \rank(N_{C'/X_{r,s}}),
\]
$\Lambda$ intersects the fiber of the subbundle $T_{X_{r,s}/\bP^{r-s-1}}|_q$ in dimension 1. Therefore, we may choose the points $v_1, \ldots, v_{r-s-2}, v$ in such a way that $\Lambda$ is equal to the span of the normal vectors at $q$ in the directions of $v_1, \ldots, v_{r-s-2}, v$, and $\Lambda\cap T_{X_{r,s}/\bP^{r-s-1}}|_q$ is spanned by the normal vector in the direction of $v$. In particular, the line through $q,v$ has degree $(1,1)$, and the span $\Lambda'\subset \Lambda$ of the normal vectors in the directions of $v_1, \ldots, v_{r-s-2}$ intersects $T_{X_{r,s}/\bP^{r-s-1}}|_q$ in the zero subspace.

Note that all modification data are tree-like. By Propositions \ref{prop:twisting} and \ref{prop:combining}, $h^0(C, N_{C/X_{r,s}}(-D))$ is bounded above by the dimension of the space of global sections of
\begin{align*}
    &N_{C'/X_{r,s}}(-D')[q \xrightarrow{+}v_1] \ldots [q \xrightarrow{+}v_{r-s-2}][q \xrightarrow{+} v]\\
    &\qquad\qquad\qquad\qquad[q \xrightarrow{-}v_1] \ldots [q \xrightarrow{-}v_{r-s-2}][q \xrightarrow{-} T_{X_{r,s}/\bP^{r-s-1}}|_{C'}]\\
    &\cong N_{C'/X_{r,s}}(-D'-(r-s-1)q)[q\xrightarrow{+} \Lambda][q \xrightarrow{-} \Lambda'+T_{X_{r,s}/\bP^{r-s-1}}]\\
    &\cong N_{C'/X_{r,s}}(-D'-(r-s-2)q)[q\xrightarrow{+} \Lambda],
\end{align*}
because $\Lambda'$ and $T_{X_{r,s}/\bP^{r-s-1}}|_q$ span $N_{C'/X_{r,s}}|_q$.

\end{proof}

We are now ready to prove Propositions \ref{prop:main_inductive} and \ref{prop:main_inductive_refinement}.

\begin{proof}[Proof of Proposition \ref{prop:main_inductive}(a)]

Suppose that 
\[
N_{d,k}=N_{C'/X_{r,s}} \cong \cO(a)^{ x} \oplus \cO(a+1)^{ z}
\]
is balanced. Let $\Lambda \subset N_{C'/X_{r,s}}|_q$ be a general subspace of dimension $r-s-1$ as in Proposition \ref{prop:bound_collide_points}. Then, $\Lambda$ is in particular transverse to the fiber over $q$ of the positive subsheaf of $N_{C'/X_{r,s}}(-D'-(r-s-2)q)$, so
\[
N_{C'/X_{r,s}}(-D'-(r-s-2)q)[q\xrightarrow{+} \Lambda]
\]
is also balanced.

We have either $x \geq r-s-1$ or $z \geq s$. If $x \geq r-s-1$, then
\[
N_{C'/X_{r,s}}[q\xrightarrow{+} \Lambda]\cong \cO(a)^{x-(r-s-1)} \oplus \cO(a+1)^{z+(r-s-1)}.
\]
Then, taking $D'\subset C'$ of degree $(a+1)-(r-s-2)+1$ in Proposition \ref{prop:bound_collide_points} shows that
\[
h^0(C, N_{C/X_{r,s}}(-D))=h^1(C, N_{C/X_{r,s}}(-D))=0,
\]
so $N_{C/X_{r,s}}(-D)$ (and therefore $N_{C/X_{r,s}}$) must satisfy interpolation. Similarly, if $z \geq s$, then taking $D'\subset C'$ of degree $(a+2)-(r-s-2)+1$ shows that $N_{C/X_{r,s}}$ satisfies interpolation. Because interpolation is open, it follows in both cases that $N_{d+(r-s-1),k+1}$ is balanced.
\end{proof}

\begin{proof}[Proof of Proposition~\ref{prop:main_inductive_refinement}]
We prove (a); (b) is similar. Write
\[
N_{d,k}=N_{C'/X_{r,s}} \cong \cO(a)^x \oplus \cO(a+1)^y \oplus \cO(a+2)^z,
\]
where $x \geq r-s-1$. Let $\Lambda \subset N_{C'/X_{r,s}}|_q$ be a general subspace of dimension $r-s-1$ as in Proposition \ref{prop:bound_collide_points}. Then, $\Lambda$ is transverse to all of the pieces of the Harder-Narasimhan filtration of $N_{C'/X_{r,s}}$, and thus
\[
N_{C'/X_{r,s}}[q\xrightarrow{+} \Lambda]\cong \cO(a)^{x-(r-s-1)} \oplus \cO(a+1)^{y+(r-s-1)} \oplus \cO(a+2)^z.
\]

Now, semi-continuity gives the following statements, which together imply Proposition \ref{prop:main_inductive_refinement}(a).
\begin{itemize}
\item[(1)] Taking $D'\subset C'$ of degree $(a+2)-(r-s-2)+1$ in Proposition \ref{prop:bound_collide_points} shows that $h^0(N_{C/X_{r,s}}(-D))=0$ whenever $D$ has degree $(a+2)+(r-s)$. Thus, $N_{d+(r-s-1),k+1}$ has no summands of degree greater than $(a+2)+(r-s-1)$
\item[(2)] Taking $D'\subset C'$ of degree $(a+2)-(r-s-2)-1$ shows that $h^1(N_{C/X_{r,s}}(-D))=0$ whenever $D$ has degree $a+(r-s)$. Thus, $N_{d+(r-s-1),k+1}$ has no summands of degree less than $a+(r-s-1)$.
\item[(3)] Taking $D'\subset C'$ of degree $(a+2)-(r-s-2)$ shows that $h^0(N_{C/X_{r,s}}(-D))\le z$ and $h^1(N_{C/X_{r,s}}(-D))\le x-(r-s-1)$ whenever $D$ has degree $(a+1)+(r-s)$. Thus, $N_{d+(r-s-1),k+1}$ has at most $z$ summands of degree exactly $(a+2)+(r-s)-1$, and at most $x-(r-s-1)$ summands of degree exactly $a+(r-s-1)$. 
\end{itemize}
\end{proof}

\begin{proof}[Proof of Proposition \ref{prop:main_inductive}(b)]

We are done by Proposition~\ref{prop:main_inductive_refinement} if $m>1$. If $m=1$, then $N_{C'/X_{r,s}}[q\xrightarrow{+} \Lambda]$ is balanced, and we conclude as in the proof of Proposition \ref{prop:main_inductive}(a).
 
\end{proof}

\subsection{Auxiliary arguments}

Larson-Vogt's proof of Theorem \ref{thm:Pr} (see \cite[\S 9.1]{lv}) uses the inductive steps of Proposition \ref{prop:lv_inductive} below, for $k=0$. The proofs go through with minimal changes for $k>0$. We will only use Proposition \ref{prop:lv_inductive} in the case $k=2$. 

\begin{proposition}\label{prop:lv_inductive}(cf. \cite[Proposition 6.1, Proposition 8.2]{lv})
\,
\begin{enumerate}
\item[(a)] Suppose that $\frac{r+2-sk}{2} \leq d \leq \frac{3r-2-sk}{2}$. If $N_{(d-1,k;r-1,s)}$ is balanced, then so is $N_{(d,k;r,s)}$.
\item[(b)] Suppose that $\frac{3r-sk}{2} \leq d \leq \frac{5r-sk-4}{2}$. If $N_{(d-2,k;r-1,s)}$ is balanced, then so is $N_{(d,k;r,s)}$.
\item[(c)] Suppose that $s < r-2$, that $\charac(\bK) \neq 2$, and that $d = \frac{3r-1-sk}{2}$. If $N_{(d-3,k;r-2,s)}$ is balanced, then so is $N_{(d,k;r,s)}$.
\end{enumerate}
\end{proposition}

In parts (a) and (b), we allow $s = r-2$. In this case, we interpret the statement that $N_{(d',k;r-1,r-2)}$ is balanced as the statement that a general rational curve of degree $d'$ in $\bP^{r-1}$ has balanced normal bundle. Similarly, we allow $s = r-3$ in (c).

\begin{proof}
Let $f:C=C\to X_{r,s}$ be a general map of degree $(d,k)$.
\begin{enumerate}
\item[(a)] Projecting to $X_{r-1,s}$ from a general point $p \in C$ induces a pointing bundle exact sequence \cite[(8)]{lv}
\[
0 \to N_{C \to p} \to N_{C/X_{r,s}} \to N_{\overline{C}/X_{r-1,s}}(p) \to 0.
\]
where $\overline{C}$ denotes the curve $C$ viewed in $X_{r-1,s}$. The bundle on the right is balanced of degree $rd-(r-s-2)k-4$ and rank $r-2$, and the term on the left is a line bundle of degree $d-k+2$. By assumption,
\[
d-k+1 \leq \frac{rd-(r-s-2)k-4}{r-2} \leq d-k+3,
\]
so the sequence splits, and $N_{d,k}$ is balanced.
\item[(b)] Replace $C$ with a union $C' \cup L$, where $L$ is a line of degree $(1,0)$. Let $x = C' \cap L$, and let $y \in L$ be a general point. By the proof of \cite[Lemma 8.5]{aly} (or Lemma \ref{lem: criterion interpolation}), it suffices to show that $N_{C'}[2x \overset{+}{\to} y]$ is balanced. Specialize $y$ to a point $p \in C'$. Project from $p$, we obtain the pointing bundle exact sequence
\[
0 \to N_{C' \to p}(2x) \to N_{C'/X_{r,s}}[2x \overset{+}{\to} p] \to N_{\overline{C}'/X_{r-1,s}}(p) \to 0.
\]
The bundle on the right is balanced of degree $rd-(r-s-2)k-(r+4)$ and rank $r-2$, and the term on the left is a line bundle of degree $d-k+3$. Conclude if
\[
d-k+2 \leq \frac{rd-(r-s-2)k-(r+4)}{r-2} \leq d-k+4.
\]

\item[(c)] Replace $C$ with $C' \cup L$, where $L$ is a line of degree $(1,0)$. Let $x = C' \cap L$ and $y \in L$ be a general point. It suffices to show that $N_{C'}[2x \overset{+}{\to} y]$ is balanced. Let $u,v \in C$ be general points, specialize $y$ onto the line $\overline{uv}$, and project from $\overline{uv}$. We obtain
\[
0 \to \left(N_{C' \to u} \oplus N_{C' \to v}\right)[2x \overset{+}{\to} y] \to N_{C'/X_{r,s}}[2x \overset{+}{\to} y] \to N_{\overline{C}'/X_{r-2,s}}(u+v) \to 0.
\]
The bundle on the right is isomorphic to $\cO(d-k+2)^{ r-3}$, and under the assumption that $\charac(\bK) \neq 2$, the term on the left is isomorphic to $\cO(d-k+2)^{ 2}$ (cf. \cite[Lemma 6.3]{lv}).
\end{enumerate}

\end{proof}

\section{Normal bundles in characteristic not 2}\label{sec:char_not2}

In this section, we determine all normal bundles $N_{(d,k;r,s)}$ assuming $\charac(\bK)\neq2$. We deduce the answer to Question \ref{q:interpolation}(a) in \emph{all} characteristics in \S\ref{sec:interpolation_normal}. 

Most of the remaining work is combinatorial, exploiting the tension between the results proved in \S \ref{sec:obstructions} and \S\ref{sec:inductive}. As such, throughout this section and the next (dealing with characteristic 2), we will sometimes abuse terminology, by saying \((d,k)\) or $(d,k;r,s)$ is balanced or $m$-almost balanced, to mean that \(N_{d,k}\) is so. 

\begin{theorem}\label{thm:main_char_not2}
Suppose that $\charac(\bK)\neq2$.
\begin{enumerate}
\item[(a)] If $(d,k)$ is non-degenerate and if $d\ge L_{r,s}(k)$, then $N_{(d,k;r,s)}$ is balanced. 
\item[(b)] (Proposition \ref{prop:projection_split_char_not2}) If $(d,k)$ is non-degenerate and if $d<L_{r,s}(k)$, then $N_{(d,k;r,s)}$ is not balanced, and $N_{(d,k;r,s)}\cong T_{\rel}\oplus N_{C/\bP^{r-s-1}}$. The summands are determined by Lemma \ref{lem:Trel_balanced} and Theorem \ref{thm:Pr}, respectively. 
\item[(c)] If $(d,k)$ is degenerate, then $N_{(d,k;r,s)}$ is determined by Proposition \ref{prop:degenerate_splittings}.
\end{enumerate}
\end{theorem}

Only (a) remains to be proven. We proceed by induction on $r$ and $k$. More precisely, we prove the claim for $(d,k;r,s)$, assuming that it holds for all $(d',k';r',s')$ with either $r' < r$, or $(r',s') = (r,s)$ and $k' < k$. Assume throughout this section that $\charac(\bK) \neq 2$. 

\subsection{Small $k$}

\begin{lemma}\label{lem:k_zero_or_one}
Theorem~\ref{thm:main_char_not2}(a) holds when $k=0,1$.
\end{lemma}

\begin{proof}
In both cases, the assumption $d \geq L_{r,s}(k)$ is vacuous, given the non-degeneracy hypothesis $d \geq r$.

The case $k=0$ is Theorem~\ref{thm:Pr}(a). The case $k=1$ follows from the case $k=0$, because $N_{d,1}$ is obtained from $N_{d,0}$ by applying an elementary modification with respect to a general subspace of the fiber at a general point.
\end{proof}

\begin{lemma}\label{lem:k_equals_2_s_equals_0}
Theorem~\ref{thm:main_char_not2}(a) holds when $k=2$ and $s=0$.
\end{lemma}

\begin{proof}
In this case, the assumption $d \geq L_{r,s}(k)$ is that $d \geq \frac{3}{2}r$. If $d\ge 2r-1$, then $(d-(r-1),1)$ is balanced by Lemma \ref{lem:k_zero_or_one}, and conclude by Proposition \ref{prop:main_inductive}(a). If $d\le U_{r,0}(2)=\frac{5r-4}{2}$, then $(d,2)$ is balanced by Proposition \ref{prop:projection_split_char_not2}.
\end{proof}

\begin{lemma}\label{lem:k_equals_2_s_greater-0}
Theorem~\ref{thm:main_char_not2}(a) holds when $k=2$ and $s>0$.
\end{lemma}

\begin{proof}
If $k=2$ and $s>0$, then the assumption $d \geq L_{r,s}(k)$ is vacuous. We fix $s$ and proceed by induction on $r$. If $r=s+1$, then we interpret the statement that $(d,2;r,s)$ is balanced as the statement that a general rational curve of degree $d \geq r$ has balanced normal bundle in $\bP^r$. If $r=s+2$, then $(r-2,0)$ is $2$-almost balanced and $(r-1,0)$ is $1$-almost balanced by Theorem~\ref{thm:Pr}(a). Furthermore, $(d,0)$ is balanced for $d \geq r$ by Theorem~\ref{thm:Pr}(a). It follows from Proposition~\ref{prop:main_inductive} that $(d,2)$ is balanced for all $d \geq r$.

Now, assume that $r \geq s+3$ and $(d,2;r',s)$ is balanced whenever $d\ge r'$ and $r' \in [s+1,r)$, and consider $(d,2;r,s)$. If $d \geq 2r-s-1$, then $(d-(r-s-1),1)$ is balanced by Lemma~\ref{lem:k_zero_or_one}. Thus, we may conclude from Proposition~\ref{prop:main_inductive} that $(d,2)$ is balanced. If $d < 2r-s-1$, then we have in particular that $r \leq d \leq \frac{5r-2s-4}{2}$. In this case, we may apply one of the three cases of Proposition~\ref{prop:lv_inductive}, unless $d=r$ and $d \geq \frac{3r-1-2s}{2}$, in which case the inductive hypothesis does not give that $(d-2,k;r-1,s)$ and $(d-3,k;r-2,s)$ are balanced.

On the other hand, if $d=r \geq \frac{3r-1-2s}{2}$, or equivalently $r\le 2s+1$, then $(r-2(r-s-1),0)=(2s-r+2,0)$ is 2-almost-balanced in $X_{r,s}$. Indeed,
\[
N_{2s-r+2,0}\cong \cO(2s-r+2)^{ 2(r-s-1)}\oplus \cO(2s-r+4)^{ 2s-r+1},
\]
and $\frac{2(r-s-1)}{r-s-1},\frac{2s-r+1}{s}\le 2$. We conclude again by Proposition~\ref{prop:main_inductive}.
\end{proof}

\subsection{Non-degenerate predecessor}

We assume henceforth that $k \geq 3$. If $d \leq U_{r,s}(k)$, then we are done by Proposition~\ref{prop:projection_split_char_not2}(b). Therefore, we may assume that $d > U_{r,s}(k)$. If $d \geq L_{r,s}(k)$, then we will deduce balancedness of $(d,k)$ from (almost-)balancedness of the predecessor $(d-(r-s-1),k-1)$. 

\begin{lemma}\label{lem:U_L_gap}
We have
\[
U_{r,s}(k) \geq L_{r,s}(k-1)+(r-s-1).
\]
\end{lemma}

\begin{proof}
The needed inequality is equivalent to
\[
\frac{(r-s)(k+1)+(r-s-2)\left\lfloor \frac{k}{s+1} \right\rfloor}{2}
\geq
\frac{(r-s)(k-2)+4+(r-s-2)\left\lceil \frac{k-1}{s+1} \right\rceil}{2}
+(r-s-1),
\]
which in turn is equivalent to
\[
(r-s-2)\left(\left\lfloor \frac{k}{s+1} \right\rfloor-\left\lceil \frac{k-1}{s+1} \right\rceil+1\right) \geq 0.
\]
\end{proof}

In particular, if the predecessor $(d-(r-s-1),k-1)$ is itself non-degenerate, then it is balanced by the inductive hypothesis. By Proposition~\ref{prop:main_inductive}, we may conclude that $(d,k)$ is balanced.

\subsection{Degenerate predecessor}

It now suffices to consider the following situation.
\begin{situation}\label{sit:degenerate_predecessor}
Consider $r,s,d,k$ such that:
\begin{itemize}
\item $k \geq 3$,
\item $d > U_{r,s}(k)$,
\item $(d,k)$ is non-degenerate, and
\item $(d-(r-s-1),k-1)$ is degenerate.
\end{itemize}
\end{situation}

Write $d':=d-(r-s-1)$ and $k':=k-1$.

\begin{lemma}\label{lem:degenerate_case_3}
In Situation~\ref{sit:degenerate_predecessor}, we have $d'-k'\ge r-s-1$. That is, case (3) of Situation~\ref{sit:degenerate} holds for $(d',k')$.
\end{lemma}

\begin{proof}
The assumption that $d > U_{r,s}(k)$ is equivalent to
\begin{equation}\label{eq:d>U_equiv}
2(d-k-1) > (r-s-2)\left(k+1+\left\lfloor \frac{k}{s+1} \right\rfloor\right).
\end{equation}
If $k \geq 3$, then the right-hand side is at least $4(r-s-2)$. Therefore, we have $2(d-k-1) \ge 4(r-s-2)$, which is equivalent to $d'-k'\ge r-s-1$.
\end{proof}

Write $s':=s-(r-d')=d-2r+2s+1$, so that $d'-s'=r-s$. By Proposition~\ref{prop:degenerate_splittings}(3), we deduce that 
\[
N_{(d',k';r,s)}\cong N_{(d',k';d',s')}\oplus\cO(d')^{r-d'}.
\]

\begin{lemma}\label{lem:degenerate_in_balanced_range}
With notation as above, assume that the conclusion of Theorem~\ref{thm:main_char_not2}(a) holds for all $r' < r$. Then, $N_{(d',k';d',s')}$ is balanced.
\end{lemma}

\begin{proof}
It suffices to show that $d' \geq L_{d',s'}(k')$. This is equivalent to
\[
d' \geq \frac{(d'-s')(k'-1)+4+(d'-s'-2)\left\lceil \frac{k'}{s'+1} \right\rceil}{2}.
\]
By Lemma~\ref{lem:degenerate_case_3} and the assumption $k\ge 3$, we have $\left\lceil \frac{k'}{s'+1} \right\rceil = 1$. Therefore, the needed inequality is
\[
2(d-(r-s-1)) \geq (r-s)(k-2)+4+(r-s-2).
\]
This inequality simplifies to $2d \geq (r-s)(k+1)$, which is immediate from \eqref{eq:d>U_equiv}.
\end{proof}

\begin{lemma}\label{lem:degenerate_almost_balanced}
With notation as above, $N_{(d',k';r,s)}\cong N_{(d',k';d',s')}\oplus\cO(d')^{r-d'}$ is almost balanced.
\end{lemma}

\begin{proof}
We claim first that $d' \leq \mu\left(N_{(d',k';d',s')}\right) \leq d'+2$. Indeed, we compute
\[
\deg\left(N_{(d',k';d',s')}\right) = (d'+1)d'-(r-s-1)k'-2 \leq (d'-1)(d'+2),
\]
and $\rk\left(N_{(d',k';d',s')}\right)=d'-1$. On the other hand, we have
\[
\deg\left(N_{(d',k';d',s')}\right) \geq (d'-1)d'
\qquad \Longleftrightarrow \qquad
2d-(r-s)(k+1)+(k-1) \geq 0,
\]
which follows from \eqref{eq:d>U_equiv}.

If $d' \leq \mu\left(N_{(d',k';d',s')}\right) \leq d'+1$, then we are done by Lemma \ref{lem:degenerate_in_balanced_range}. Otherwise, assume that $d'+1 <\mu\left(N_{(d',k';d',s')}\right) \leq d'+2$. Then, by Lemma \ref{lem:degenerate_in_balanced_range},
\[
N_{(d',k';r,s)}\cong \cO(d')^{r-d'} \oplus \cO(d'+1)^{d'-1-\delta} \oplus \cO(d'+2)^{\delta},
\]
where 
\[
\delta=\deg\left(N_{(d',k';d',s')}\right)-(d'+1)(d'-1)=d-(r-s-1)k-1.
\]
We need to show that $r-d'\le r-s-1$ and that $\delta\le s$. The inequality $r-d'\le r-s-1$ is equivalent to $d\ge r$, which holds by the non-degeneracy assumption. The inequality $\delta\le s$ holds because 
\[
\delta=d-(r-s-1)k-1\le (r+(r-s-1))-2(r-s-1)-1=s.
\]
\end{proof}

We conclude by Proposition~\ref{prop:main_inductive} that $N_{(d,k;r,s)}$ is balanced. This completes the proof of Theorem \ref{thm:main_char_not2}(a), and therefore of Theorem \ref{thm:intro_char_not2}.

\begin{remark}\label{rem:both_arguments}
The application of Proposition~\ref{prop:main_inductive} breaks down when $k=2$. For example, on $\Bl_{\bP^2}(\bP^7)$, we have, in arbitrary characteristic,
\[
N_{4,1} \cong \cO(3) \oplus \cO(4)^2 \oplus \cO(5)^3.
\]
In particular, $N_{4,1}$ is not almost balanced, so we cannot conclude from Proposition~\ref{prop:main_inductive} that $N_{8,2}$ is balanced. This is for good reason, as Corollary~\ref{cor:char2_obstruction} shows that $N_{8,2}$ is not balanced in characteristic $2$. For this reason, our proof that $N_{8,2}$ is balanced when $\charac(\bK)\neq2$ passes instead through Proposition~\ref{prop:lv_inductive}(c), which requires this hypothesis on the characteristic.

On the other hand, Proposition~\ref{prop:lv_inductive} by itself is also not sufficient to prove Theorem~\ref{thm:main_char_not2}(a), due to the restrictive upper bounds on $d$.
\end{remark}

\subsection{Interpolation revisited}\label{sec:interpolation_normal}

Recall from \S\ref{sec:bundles} that the answer to Question \ref{q:interpolation}(a) is determined in characteristic 0 by Theorem \ref{thm:main_char_not2}. Namely, the map 
\[
\rho \colon \cM_{0,n}(X_{r,s},(d,k)) \to (X_{r,s})^n
\]
is dominant if and only if $n$ is at most 1 more than the smallest degree summand of $N_{d,k}$. In fact, the splitting of $N_{d,k}$ in characteristic 0 decides the dominance of $\rho$ \emph{in any characteristic}, including 2.

\begin{proposition}\label{prop:rho_char_indep}
Suppose that $\operatorname{char}(\bK) = p > 0$. Then, the evaluation morphism $\rho$ is dominant over $\mathbb{C}$ (or equivalently over $\mathbb{Q}$) if and only if it is dominant over $\bK$.
\end{proposition}

Viewing $X_{r,s}$ and $\cM_{0,n}(X_{r,s},(d,k))$ as defined over $\bZ$, the proof relies on the existence of a suitable compactification $\overline{\cM}$ of the moduli space $\cM_{0,n}(X_{r,s},(d,k))$ that satisfies the following three properties:
\begin{enumerate}
    \item[$\bullet$] It is defined as a proper scheme over $\operatorname{Spec}(\mathbb{Z})$.
    \item[$\bullet$] The evaluation map extends to a proper global morphism $\bar{\rho} \colon \overline{\cM} \to (X_{r,s})^n$ over $\mathbb{Z}$.
    \item[$\bullet$] $\overline{\cM}$ is irreducible of expected dimension, as are its fibers over every $(p)\in \mathbb{Z}$ (including $p=0$). In particular, $\overline{\cM}$ contains $\cM=\cM_{0,n}(X_{r,s},(d,k))$ as a dense open subscheme.
\end{enumerate}
Such a compactification is provided by the moduli stack 
of basic $n$-marked genus $0$ stable logarithmic maps to $X_{r,s}$ of degree $(d,k)$ with transverse tangency conditions \cite[Corollary~8]{RW20}.

\begin{proof}
Note also that $\rho$ is dominant over $\bK$ if and only if it is so over $\mathbb{F}_p$.  Consider the global evaluation map $\bar{\rho} \colon \overline{\cM} \to (X_{r,s})^n$, and let $W\subset (X_{r,s})^n$ be the image of $\bar{\rho}$, which must be an irreducible closed subscheme.

Assume first that $\rho$ is dominant upon restriction to the characteristic $p$ fiber. Then, $W$ contains the entire characteristic $p$ fiber of $(X_{r,s})^n$, which is an irreducible divisor in $(X_{r,s})^n$. On the other hand, $W$ meets all other fibers of $(X_{r,s})^n\to \operatorname{Spec}(\bZ)$ non-trivially. Therefore, $\bar{\rho}$ must be surjective. Because the open locus $\cM\subset \overline{\cM}$ is dense in every fiber over $\operatorname{Spec}(\bZ)$, we conclude that $\rho$ is dominant in characteristic 0.

Conversely, if $W$ contains a dense open subset of the characteristic 0 fiber of $(X_{r,s})^n$, then it is itself dense, hence $W=(X_{r,s})^n$ once more.
\end{proof}

\section{Normal bundles in characteristic 2}\label{sec:char2}

In this section, we determine all normal bundles  $N_{(d,k;r,s)}$ assuming $\charac(\bK)=2$. First, we summarize the necessary and sufficient conditions for balancing, giving a more precise version of Theorem \ref{thm:intro_char_2}.

\begin{theorem}\label{thm:main_char2}
Suppose (here and throughout this section) that $\charac(\bK) = 2$. Then, $(d,k;r,s)$ is balanced if and only if both of the following conditions hold.
\begin{enumerate}
\item[(i)] Let $\ell \in [0,r-2]$ be the residue of $d-1 \pmod{r-1}$. Then, either $\ell=0$, or $s \neq 0$ and
\begin{equation}\label{eq:k_lower_bound}
k \geq \left\lceil\frac{\ell}{r-s-1}\right\rceil, \left\lceil\frac{r-1-\ell}{s}\right\rceil
\end{equation}
\item[(ii)] Let $\epsilon \in [0,r-s-3]$ be the residue of $d-k-1 \pmod{r-s-2}$. Then, either $\epsilon=0$ and $d \geq L_{r,s}(k)$, or $\epsilon \neq 0$ and both of the following inequalities hold:
\begin{align}
\frac{s+2}{s+1} \cdot k - 1 &< 2\left\lfloor \frac{d-k-1}{r-s-2} \right\rfloor, \label{eq:no_splitting} \\
d+(s+1)k-s-2 &\leq (r+s)\left\lfloor \frac{d-k-1}{r-s-2} \right\rfloor. \label{eq:no_low}
\end{align}
\end{enumerate}
If $s=r-2$, then condition (ii) is considered to be vacuously true.
\end{theorem}

Recall that (i) is necessary by Corollary~\ref{cor:char2_obstruction}. If $\epsilon=0$, then $d \geq L_{r,s}(k)$ is necessary by Proposition~\ref{prop:projection_split_char_not2}. If $\epsilon \neq 0$, then \eqref{eq:no_splitting} and \eqref{eq:no_low} are necessary by Propositions~\ref{prop:projection_split_char2} and \ref{prop:projection_char2_low_obstruction}, respectively. It therefore suffices to prove the ``if'' direction of Theorem \ref{thm:main_char2}.

In fact, we prove more, determining for all $(d,k;r,s)$ the normal bundle $N_{d,k}$. We will decrement
\[
(d,k) \mapsto (d-(r-s-1),k-1),
\]
repeatedly until reaching $(d',k')$ for which $N_{d',k'}$ is determined, either by Theorem~\ref{thm:main_char2} $(k'=0)$, or by Proposition~\ref{prop:projection_split_char_not2} or \ref{prop:projection_split_char2}, depending on $\epsilon$. We consider three cases:
\begin{enumerate}
\item[(i)] $d > k(r-s-1)$ (automatic if $s=r-2$)
\item[(ii)] $d \le k(r-s-1)$ and  $(r-s-2)\nmid (d-k-1)$
\item[(iii)] $d \le k(r-s-1)$ and  $(r-s-2)\mid (d-k-1)$
\end{enumerate}

The normal bundles $N_{d,k}$ are determined in these three cases in Theorems \ref{thm:d>k(r-s-1)}, \ref{thm:d<k(r-s-1)_epsilon_not0}, and \ref{thm:d<k(r-s-1)_epsilon_zero}, respectively.

\subsection{$d > k(r-s-1)$}

\begin{theorem}\label{thm:d>k(r-s-1)}
Assume that $d > k(r-s-1)$ and that $\charac(\bK)=2$. As in Theorem~\ref{thm:Pr}, write $\ell \in [0,r-2]$ for the residue of $d-1 \pmod{r-1}$, and
\[
\alpha = d+2 \cdot  \frac{d-1-\ell}{r-1},
\]
so that
\[
N_{d,0} \cong \cO(\alpha)^{r-1-\ell} \oplus \cO(\alpha+2)^\ell.
\]
\begin{enumerate}
\item[(a)] Suppose that $k < \left\lceil \frac{\ell}{r-s-1} \right\rceil$. Then,
\[
N_{d,k} \cong \cO(\alpha)^{r-1-\ell} \oplus \cO(\alpha+1)^{k(r-s-1)} \oplus \cO(\alpha+2)^{\ell-k(r-s-1)}.
\]
\item[(b)] Suppose that $k < \left\lceil \frac{r-1-\ell}{s} \right\rceil$ (or that $s=0$). Then,
\[
N_{d,k} \cong \cO(\alpha-k)^{(r-1-\ell)-ks} \oplus \cO(\alpha-k+1)^{ks} \oplus \cO(\alpha-k+2)^\ell.
\]
\item[(c)] Suppose that $k \geq \left\lceil \frac{\ell}{r-s-1} \right\rceil, \left\lceil \frac{r-1-\ell}{s} \right\rceil$. Then, $N_{d,k}$ is balanced.
\end{enumerate}
\end{theorem}

\begin{proof}
Assume first that $k < \left\lceil \frac{\ell}{r-s-1} \right\rceil$. There is nothing to check if $k=0$, so assume that $k\ge 1$. Write $d' = d-k(r-s-1)$. Then,
\[
\ell':=\ell-k(r-s-1) \equiv d'-1\pmod{r-1},
\]
and $\ell' \in [0,r-2]$. In particular, we have
\[
N_{d',0} \cong \cO(\alpha')^{r-1-\ell'} \oplus \cO(\alpha'+2)^{\ell'},
\]
where
\begin{align*}
\alpha' &= d' + 2 \cdot \frac{d'-1-\ell'}{r-1} \\
&=\alpha-k(r-s-1).
\end{align*}

By Proposition~\ref{prop:main_inductive_refinement}(a), we have
\begin{equation}\label{eq:Ndk_shape}
N_{d,k} \cong \cO(\alpha)^{x'} \oplus \cO(\alpha+1)^{y'} \oplus \cO(\alpha+2)^{z'},
\end{equation}
where $x' \leq (r-1-\ell')-k(r-s-1)=r-1-\ell$ and $z' \leq \ell'=\ell-k(r-s-1)$. On the other hand, by Lemma \ref{lem:char2_obstruction_general}(a), we also have $z' \geq \ell-k(r-s-1)$, so equality must hold. By comparing degrees of both sides of \eqref{eq:Ndk_shape}, we must also have $x'=r-1-\ell$ and $y'=k(r-s-1)$. This proves (a). The proof of (b) is similar, using instead Proposition~\ref{prop:main_inductive_refinement}(b) and Lemma \ref{lem:char2_obstruction_general}(b).

Finally, we prove (c). Assume that $\ell \leq k(r-s-1)$ and $r-1-\ell \leq ks$. (In particular, $s\neq 0$.) As above, let $\ell' \in [0,r-2]$ be the residue of $d'-1 \pmod{r-1}$, so that
\[
N_{d',0} \cong \cO(\alpha')^{r-1-\ell'} \oplus \cO(\alpha'+2)^{\ell'}.
\]
If $\ell'=0$, then we are done by Proposition~\ref{prop:main_inductive}(a), so assume that $\ell'>0$. By Proposition~\ref{prop:main_inductive}(b), it suffices to show that $\ell' \leq sk$ and $r-1-\ell' \leq k(r-s-1)$.

Write $\ell+ks = j(r-1)+\ell'$. Then, we have
\[
\ell+ks \leq k(r-s-1)+ks = k(r-1)
\Rightarrow j \leq k-1,
\]
because we assume $\ell'>0$. Therefore,
\begin{align*}
\ell' &= \ell+ks-j(r-1) \\
&\geq \ell+ks-(k-1)(r-1) \\
&= \ell-k(r-s-1)+(r-1) \\
&\geq (r-1)-k(r-s-1),
\end{align*}
hence $r-1-\ell' \leq k(r-s-1)$.

For the other inequality, note first that
\[
r-1-\ell \leq ks \Longleftrightarrow \ell+ks \geq r-1 \Rightarrow j \geq 1.
\]
Therefore,
\begin{align*}
\ell' &= \ell+ks-j(r-1) \\
&\leq \ell+ks-(r-1) \\
&\leq ks,
\end{align*}
as needed.
\end{proof}

\subsection{$d \le k(r-s-1)$ and $(r-s-2)\nmid (d-k-1)$}

\begin{theorem}\label{thm:d<k(r-s-1)_epsilon_not0}
Suppose that $d\le k(r-s-1)$ and that we are in Situation \ref{sit:char2_notdiv}. 
\begin{enumerate}
\item[(a)] (Proposition \ref{prop:projection_split_char2}) If $\mu(T_{\rel}) \ge a+1$, then $N_{d,k} \cong T_{\rel} \oplus N_{C/\bP^{r-s-1}}$. The summands are determined by Lemma \ref{lem:Trel_balanced} and Theorem \ref{thm:Pr}, respectively, and $N_{d,k}$ is not balanced.
\item[(b)] (cf. Proposition \ref{prop:projection_char2_low_obstruction}) If $\mu(T_{\rel}) < a+1$ and 
\[
z:=d+(s+1)k-s-2 - (r+s)\left\lfloor \frac{d-k-1}{r-s-2} \right\rfloor>0, 
\]
then 
\[
N_{d,k} \cong \cO(a)^{r-s-2-\epsilon} \oplus \cO(a+1)^{y} \oplus \cO(a+2)^{z},
\]
where $y=(r-1)-(r-s-2-\epsilon)-z$. In particular, $N_{d,k}$ is not balanced.
\item[(c)] If $\mu(T_{\rel}) < a+1$ and 
\[
z:=d+(s+1)k-s-2 - (r+s)\left\lfloor \frac{d-k-1}{r-s-2} \right\rfloor\le 0, 
\]
then $N_{d,k}$ is balanced.
\end{enumerate}
\end{theorem}

Only (b) and (c) require proofs. We first show that, in the setting of (b) and (c), we can ignore the case $s=0$, in addition to the case $s=r-2$.
\begin{lemma}
Suppose that $s=0$, so $d \le (r-1)k$. Then, $\mu(T_{\rel}) \geq a+1$, so we are in case (a) of Theorem~\ref{thm:d<k(r-s-1)_epsilon_not0}.
\end{lemma}

\begin{proof}
We have
\[
\begin{array}{r@{\quad}r@{\;}c@{\;}l}
& \displaystyle d &<& \displaystyle (r-1)k+1 \\
\Leftrightarrow & \displaystyle (r-2)k &> & \displaystyle d-k-1 \\
\Leftrightarrow & \displaystyle k &> & \displaystyle \left\lfloor \frac{d-k-1}{r-2} \right\rfloor \\
\Leftrightarrow & \displaystyle d+k &\ge & \displaystyle d-k+2\cdot\left\lfloor \frac{d-k-1}{r-2} \right\rfloor+1 \\
\Leftrightarrow & \displaystyle \mu(T_{\rel}) &\geq & \displaystyle a+1.
\end{array}
\]
\end{proof}

Suppose that $\mu(T_{\rel}) < a+1$, so that $s\neq0$. The decrement
\[
(d,k) \mapsto (d-(r-s-1),k-1)
\]
has the effect of decreasing $\mu(T_{\rel})$ by $r-s-\frac{s}{s+1}$, and decreasing $a$ by $r-s$. Note also that this decrement leaves the value of $\epsilon$ unchanged. As $s\neq 0$, decrementing
\[
m := \left\lceil \frac{(s+1)(a+1)-\left((s+1)d+k\right)}{s} \right\rceil
\]
times results in a degree $(d',k') := (d-m(r-s-1),k-m)$ with the property that
\[
a'+1 \leq \mu(T'_{\rel}) < a'+2.
\]
Here, $T'_{\rel}$, $a'$ denote the relative tangent bundle and integer, respectively, associated to $(d',k')$.

In order for this to be sensible, we require the following.
\begin{lemma}\label{lem:decrement_sensible}
With notation as above, we have $d' > k' \geq 0$.
\end{lemma}

\begin{proof}
Write $q=\left\lfloor \frac{d-k-1}{r-s-2} \right\rfloor$, so that $a=d-k+2q$. We have
\[
\renewcommand{\arraystretch}{1.25}
\begin{array}{r@{\quad}r@{\;}c@{\;}l}
& \displaystyle k(r-s-1) &\geq& \displaystyle d \\
\Leftrightarrow & \displaystyle k(r-s-2) &>& \displaystyle d-k-1 \\
\Leftrightarrow & \displaystyle k &>& \displaystyle q \\
\Leftrightarrow & \displaystyle 2k &\geq& \displaystyle 2q+1 \\
\Leftrightarrow & \displaystyle sk &\geq& \displaystyle (s+1)(-k+2q+1)-k \\
\Leftrightarrow & \displaystyle k &\geq& \displaystyle \left\lceil \frac{(s+1)(d-k+2q+1)-(s+1)d-k}{s} \right\rceil \\
\Leftrightarrow & \displaystyle k &\geq& \displaystyle m \\
\Leftrightarrow & \displaystyle k' &\geq& \displaystyle 0.
\end{array}
\]
Moreover, because $k \geq q+1$, we have
\begin{align*}
m
&= \left\lceil \frac{(s+1)(-k+2q+1)-k}{s} \right\rceil \\
&\leq \left\lceil \frac{(s+1)(2q+1)-(s+2)(q+1)}{s} \right\rceil \\
&= \left\lceil \frac{sq-1}{s} \right\rceil \\
&\leq q.
\end{align*}
Therefore, $d'-k'=d-k-m(r-s-2) \geq d-k-q(r-s-2)>0$.
\end{proof}

\begin{proof}[Proof of Theorem~\ref{thm:d<k(r-s-1)_epsilon_not0}]
Because $a'+1 \leq \mu(T'_{\rel}) < a'+2$, we have
\[
N'_{d',k'} \cong T'_{\rel} \oplus N_{d'-k',0} \cong
\left(\cO(a'+1)^{y'} \oplus \cO(a'+2)^{w'}\right)
\oplus
\left(\cO(a')^{r-s-2-\epsilon} \oplus \cO(a'+2)^\epsilon\right),
\]
where
\[
w' = \deg (T'_{\rel})-(a'+1)\rk(T'_{\rel}) = (s+1)d'+k'-(a'+1)(s+1).
\]
Recall also that $a'=a-m(r-s)$. 

Now, by Proposition~\ref{prop:main_inductive_refinement}(b), if $w'+\epsilon > ms$, then
\[
N_{d,k} \cong \cO(a)^x \oplus \cO(a+1)^y \oplus \cO(a+2)^z,
\]
where $x \leq r-s-2-\epsilon$ and $z \leq w'+\epsilon-ms$. On the other hand, $N_{d,k}$ surjects onto
\[
N_{C/\bP^{r-s-1}} \cong \cO(a)^{r-s-2-\epsilon} \oplus \cO(a+2)^\epsilon,
\]
so we must also have $x \geq r-s-2-\epsilon$. By comparing degrees, it follows that, in fact,
\[
N_{d,k} \cong \cO(a)^{r-s-2-\epsilon} \oplus \cO(a+1)^y \oplus \cO(a+2)^{w'+\epsilon-ms}.
\]
If instead $w'+\epsilon \leq ms$, then by Proposition~\ref{prop:main_inductive}, $N_{d,k}$ is balanced.

Now, to complete the proofs of (b) and (c), it is enough to show that
\[
d+(s+1)k-s-2 - (r+s)\left\lfloor \frac{d-k-1}{r-s-2} \right\rfloor=w'+\epsilon-ms.
\]
Indeed, writing again $q=\left\lfloor \frac{d-k-1}{r-s-2} \right\rfloor$, we have
\begin{align*}
w'+\epsilon-ms
&= (s+1)d'+k'-(s+1)(a'+1)+\epsilon-ms \\
&= (s+1)d+k-(s+1)(a+1)+\epsilon \\
&= (s+1)d+k-(s+1)(d-k+2q+1)+\epsilon \\
&= (s+1)(k-2q-1)+k+\epsilon \\
&= (s+2)k-2(s+1)q-(s+1)+\epsilon \\
&=d+(s+1)k-s-2-(r+s)q,
\end{align*}
using that $d=k+1+q(r-s-2)+\epsilon$. This completes the proof.
\end{proof}

\subsection{$d \le k(r-s-1)$ and $(r-s-2)\mid (d-k-1)$}

\begin{theorem}\label{thm:d<k(r-s-1)_epsilon_zero}
Suppose that $d \leq k(r-s-1)$ and $(r-s-2) \mid (d-k-1)$. Write as usual $q=\frac{d-k-1}{r-s-2}$ and $a=d-k+2q$, so that
\[
N_{C/\bP^{r-s-1}} \cong \cO(a)^{r-s-2}.
\]
\begin{enumerate}
\item[(a)] (Proposition \ref{prop:projection_split_char_not2}) If $d < L_{r,s}(k) \Longleftrightarrow \mu(T_{\rel}) > a+1$, then $N_{d,k} \cong T_{\rel} \oplus N_{C/\bP^{r-s-1}}$ is not balanced. $T_{\rel}$ is determined by Lemma~\ref{lem:Trel_balanced}.
\item[(b)] If $d \ge L_{r,s}(k) \Longleftrightarrow \mu(T_{\rel}) \leq a+1$, then $N_{d,k}$ is balanced.
\end{enumerate}
\end{theorem}

\begin{proof}
We prove (b) by induction on $k$. The case $k=0$ is empty, because the hypothesis $d \leq k(r-s-1)$ cannot hold. Now, for a given $d,k$ with $k>0$, suppose that $\mu(T_{\rel}) \leq a+1$. If in addition $\mu(T_{\rel}) \geq a-1$, then we are done by Propositions~\ref{prop:projection_split_char_not2}. Assume instead that $\mu(T_{\rel}) < a-1$.

Consider now
\[
(d',k') := (d-(r-s-1),k-1),
\]
which still satisfies the hypotheses $d' \leq k'(r-s-1)$ and $(r-s-2) \mid (d'-k'-1)$. The degree $(d',k')$ is valid as long as $d-k > r-s-2$. By the hypothesis that $(r-s-2) \mid (d-k-1)$, we can only have $d-k\le r-s-2$ if $d-k=1$. However, if $d-k=1$, then 
\[
\mu(T_{\rel})=k+1+\frac{k}{s+1} \leq a+1=2
\] 
only if $k=0$, a contradiction.

Let $T'_{\rel}$ and $a'$ denote the bundle and integer associated to $(d',k')$, respectively. We have
\[
\mu(T'_{\rel})-(a'+1)=\mu(T_{\rel})-(a+1)+\frac{s}{s+1}<0,
\]
so $N_{d',k'}$ is balanced by the inductive hypothesis. We conclude by Proposition~\ref{prop:main_inductive}. 
\end{proof}

\section{Tangent bundles}\label{sec: interpolation_tangent}

In this section, we prove Theorem \ref{thm:intro_tangent}, determining the restricted tangent bundle $T_{d,k}$ of the general map $f:C\to X_{r,s}$ of degree $(d,k)$. The calculation is independent of characteristic.

Fix $r,s$ with $0 \leq s \leq r-2$. Recall \eqref{eq:proj_seq_tangent}
\[
0\to T_{\rel} \to T_{d,k} \to T_{\bP^{r-s-1}}|_{C} \to 0.
\]
\begin{proposition}\label{prop:projection_split_tangent}
\,
\begin{enumerate}
\item[(a)] The exact sequence \eqref{eq:proj_seq_tangent} splits if
\[
\mu(T_{\bP^{r-s-1}}|_{C} )=\frac{(r-s)(d-k)}{r-s-1} \leq \left\lfloor \frac{(s+1)d+k}{s+1} \right\rfloor + 1 = \left\lfloor \mu(T_{\rel}) \right\rfloor + 1.
\]
This inequality is equivalent to
\[
d \leq (r-s)(k+1)+(r-s-1)\left\lfloor \frac{k}{s+1}\right\rfloor - 1 =: U^t_{r,s}(k).
\]
\item[(b)] If $d \leq U^t_{r,s}(k)$, then $T_{d,k} \cong T_{\rel} \oplus T_{\bP^{r-s-1}}|_{C}$ is balanced if and only if
\[
\mu(T_{\bP^{r-s-1}}|_{C} )=\frac{(r-s)(d-k)}{r-s-1} \geq \left\lceil \frac{(s+1)d+k}{s+1} \right\rceil - 1 = \left\lceil \mu(T_{\rel}) \right\rceil - 1.
\]
This inequality is equivalent to
\[
d \geq (r-s)(k-1)+(r-s-1)\left\lceil \frac{k}{s+1}\right\rceil+ 1=: L^t_{r,s}(k).
\]
\end{enumerate}
In particular, if $d<L^t_{r,s}(k)$, then $T_{d,k}$ is not balanced.
\end{proposition}

\begin{proof}
$T_{\rel}$ and $T_{\bP^{r-s-1}}|_{C}$ are balanced, by Lemma \ref{lem:Trel_balanced} and Theorem \ref{thm:Pr_tangent}, respectively. The condition $d\le U^t_{r,s}(k)$ is equivalent to
    \[
    \operatorname{Ext}^1\!\bigl(
    T_{\bP^{r-s-1}}|_{C},
    \, T_{\mathrm{rel}}|_{C}
    \bigr)=0.
    \]
\end{proof}

\begin{lemma}\label{lemma: base_case_tangent}
    For every $r,s,k$ such that $0 \leq s \leq r-2$, the set of integers $d \geq k$ satisfying
    \[
    L^t_{r,s}(k) \leq d \leq U^t_{r,s}(k)
    \]
    is non-empty.
\end{lemma}
\begin{proof}
    The difference $U_{r,s}^t(k) -L_{r,s}^t(k)$ is always at least $1$ for $r-s \geq 2$.
\end{proof}

\begin{proof}[Proof of Theorem \ref{thm:intro_tangent}]
Fix \(r,s,k\). We prove by induction on \(d\) that the bundle $T_{X_{r,s}}|_C$ satisfies interpolation. The induction is initiated by Lemma \ref{lemma: base_case_tangent}. For the inductive step, degenerate $C$ to $C' \cup_p L$, where $L$ is a line of degree $(1,0)$, and apply Lemma \ref{lem: criterion interpolation} with $D=2x$ a divisor supported in $L \smallsetminus \{ p \}$. By \cite[Proposition 3.1]{larson}, we have
$$
T_{X_{r,s}}|_L \simeq \cO(1)^{ r-2} \oplus T_L ,
$$
which implies $H^0(L, T_{X_{r,s}}(-D-p))=0$. Furthermore, let
\[
V = \mathrm{ev}_{C'}^{-1}\big(\mathrm{ev}_L(H^0(L, T_{X_{r,s}}|_L(-D)))\big) .
\]
Then, $V$ can be identified with the space of sections $\sigma \in H^0(C', T_{X_{r,s}})$ such that $\sigma(p) \in T_L|_p$. Since $L$ is general, $T_L|_p \subseteq T_{X_{r,s}}|_p$ is a general line. Thus, by Lemma \ref{lem:interpolation_subspace}, $V$ satisfies interpolation. Its dimension is given by
\[
\chi(T_{X_{r,s}}|_{C'}) - (r-1)= \chi(T_{X_{r,s}|_{C'}})+ \chi(T_{X_{r,s}}|_L(-p-D)).
\]
The conclusion then follows from Lemma \ref{lem: criterion interpolation}.
\end{proof}

We deduce the answer to Question \ref{q:interpolation}(b).

\begin{corollary}
\label{cor:interpolation_tangent}
Let $p_1,\ldots,p_n \in \bP^1$ and $x_1,\ldots,x_n \in X_{r,s}$
be general points. Consider the property of the existence of a map
$f \colon \bP^1 \to X_{r,s}$ of degree $(d,k)$ satisfying $f(p_i)=x_i$
for all $i=1,\ldots,n$. Then, the largest value of $n$ for which this
property holds is
\[
n_{\max}=\min\left(
\left\lfloor \mu(T_{\bP^{r-s-1}}|_C) \right\rfloor,
\left\lfloor \mu(T_{d,k}) \right\rfloor
\right)+1.
\]
\end{corollary}

\begin{proof}
If $\charac(\bK)=0$, then $n_{\max}$ is equal to 1 more than the smallest degree summand of $T_{d,k}$. In arbitrary characteristic, the claim follows from a similar argument as in Proposition \ref{prop:rho_char_indep}.
\end{proof}

\appendix

\section{Quasimap invariants}

In the appendix, we provide a refinement of Corollary \ref{cor:interpolation_tangent}, determining the \emph{number} of interpolating maps (when finite) with respect to \emph{arbitrary} incidence conditions at general linear subspaces of $X_{r,s}$. That is, we address the following question:

\begin{question}\label{q:enumerative}
 Let \(p_1,\ldots,p_n \in \bP^1\) be general points, and let $\Lambda_1,\ldots,\Lambda_n\subset X_{r,s}$ be linear subspaces (\S\ref{sec:subspaces}) of any dimension. How many \(f \colon \bP^1 \to X_{r,s}\) of degree \((d,k)\) are there, such that $f(p_i)\in \Lambda_i$ for all $i$?
\end{question}

One should take the sum of the codimensions of the $\Lambda_i\subset X_{r,s}$ to equal the dimension of the moduli space of maps $f$ of degree $(d,k)$ in order to expect a finite answer, see Situation \ref{situation} and Proposition \ref{prop:integral_value}. In positive characteristic, the enumeration should in general be taken with possibly non-reduced multiplicities. 

We will obtain these counts as intersection numbers on moduli space of quasimaps. The techniques are parallel to those developed in \cite{cl2,lsakran}, where the primary interest was in the case in which the $\Lambda_i$ are points (``Tevelev degrees''). We refer to the reader to those works for discussion of the history of the fixed-domain curve-counting problem. While variants of the moduli space of quasimaps adapted to their respective settings were employed in \cite{cl2,lsakran}, the ordinary space of quasimaps suffices here. 

\subsection{Linear subspaces}\label{sec:subspaces}

As before, let $X_{r,s}=\Bl_{\bP^s}(\bP^r)=\bP(\cO(-1)\oplus\cO^{s+1})\to \bP^{r-s-1}$. Let $\pi:X_{r,s}\to \bP^{r-s-1}$ be the projection map and let $b:X_{r,s}\to \bP^{r}$ be the blow-up. Let $\H,\E\in A^1(X_{r,s})$ be the classes of the hyperplane (pulled back under $b$) and exceptional divisor, respectively.

We represent a point $x\in X$ by
\begin{equation*}
x=[x_0x_1:\cdots:x_0x_{r-s}:x_{r-s+1}:\cdots:x_{r+1}],
\end{equation*}
where the $x_j\in \bK$ satisfy the following non-vanishing conditions:
\begin{enumerate}
\item[(NZ1)] $x_1,\ldots,x_{r-s}$ are not all 0,
\item[(NZ2)] $x_0,x_{r-s+1},\ldots,x_{r+1}$ are not all 0,
\end{enumerate}
and taken up to simultaneous scaling of $x_1,\ldots,x_{r+1}$ and of $x_0,x_{r-s+1},\ldots,x_{r+1}$.

A \emph{linear space} $\Lambda\subset X_{r,s}$ is cut out by $r-s-1-\ell_1$ independent linear equations in $x_1,\ldots,x_{r-s}$ and $s+1+\ell_1-\ell_2$ additional independent linear equations in all of $x_0x_1,\ldots,x_{r+1}$, see also \cite[\S 2.6]{cl_complete}. The dimension of $\Lambda$ is $\ell_2$ and the dimension of its image in $\bP^{r-s-1}$ is $\ell_1$. The ordered pair $(\ell_1,\ell_2)$ is referred to as the \emph{bi-dimension}.

Let $\Gamma(\ell_1,\ell_2)$ be the variety parametrizing linear spaces of bi-dimension $(\ell_1,\ell_2)$. The natural morphism $\Gamma(\ell_1,\ell_2)\to \Gr(\ell_1+1,r-s)$ remembering the image in $\bP^{r-s-1}$ has the structure of an open subset of a Grassmannian bundle, so in particular $\Gamma(\ell_1,\ell_2)$ is smooth and quasi-projective. Let $\Lambda^{\univ}\subset\Gamma(\ell_1,\ell_2)\times X$ be the universal subscheme, proper and flat of relative dimension $\ell_2$ over $\Gamma(\ell_1,\ell_2)$. 

\begin{lemma}\label{lem:point_imposes_expected_conditions_on_Gamma}
Let $x\in X_{r,s}$ be any point. Then, the locus in $\Gamma(\ell_1,\ell_2)_x\subset \Gamma(\ell_1,\ell_2)$ of linear spaces containing $x$ is non-empty and smooth of codimension $r-\ell_2$, the expected.
\end{lemma}

\begin{proof}
It is immediate that the composition 
\begin{equation*}
\Gamma(\ell_1,\ell_2)_x\subset \Gamma(\ell_1,\ell_2)\to\Gr(\ell_1+1,r-s)
\end{equation*}
is itself an open subset of a Grassmannian subbundle over its image, which is identified with $\Gr(\ell_1,r-s-1)$.
\end{proof}

\subsection{Quasimaps}

Let $Q=Q_\beta(\bP^1,X_{r,s})$ be the moduli space of quasimaps of degree $\beta=(d,k)$ \cite{cfk}, whose precise definition is recalled below. $Q$ contains the moduli space of maps $f:\bP^1\to X_{r,s}$ as a dense open subset, allowing maps to degenerate to ``maps with base-points.''

Assume $d\ge k\ge0$. Explicitly,
\begin{equation*}
Q=\bP (H^0(\cO(-1)\otimes H^0(\bP^1,\cO_{\bP^1}(k)) \oplus H^0(\bP^1,\cO_{\bP^1}(d))^{s+1})) \to \bP (H^0(\bP^1,\cO_{\bP^1}(d-k))^{r-s})
\end{equation*}
Here, $\cO(-1)$ denotes the tautological sub-line bundle $\bP(H^0(\bP^1,\cO(d-k))^{r-s})$. Let $\pi:Q\to \bP(H^0(\bP^1,\cO_{\bP^1}(d-k))^{r-s}$ denote the projection, and let $\cO_Q(1)$ denote the relative hyperplane line bundle on $Q$.

A point of $Q$ is given by sections
\begin{itemize}
\item $u_0\in H^0(\bP^1,\cO(k))$,
\item $u_1,\ldots,u_{r-s}\in H^0(\bP^1,\cO(d-k))$,
\item $u_{r-s+1},\ldots,u_{r+1}\in H^0(\bP^1,\cO(d))$
\end{itemize}
up to the same scaling factors as points $x\in X_{r,s}$, satisfying (NZ1) and (NZ2) (where the role of the scalars $x_j$ is replaced by the sections $u_j$).

Let $\cM\subset Q$ be the dense open subset where the sections $u$ satisfy both the \emph{base-point-free conditions}:
\begin{enumerate}
\item[(BPF1)] For all $p\in\bP^1$, the sections $u_1,\ldots,u_{r-s}\in H^0(\bP^1,\cO(d-k))$ do not all vanish at $p$.
\item[(BPF2)] For all $p\in\bP^1$, the sections $u_0\in H^0(\bP^1,\cO(k))$ and $u_{r-s+1},\ldots,u_{r+1}\in H^0(\bP^1,\cO(d))$ do not all vanish at $p$.
\end{enumerate}
We also say that $u\in Q$ satisfies (BPF1) or (BPF2) at a \emph{given} point $p$ if the non-vanishing holds at that particular point. $\cM$ is identified with the space of maps $f:\bP^1\to X_{r,s}$ of degree $\beta$. In particular, for any point $p\in\bP^1$, there is an evaluation map $\ev_p:\cM\to X_{r,s}$ defined by $\ev_p(f)=f(p)$. We say that $u\in Q$ is \emph{bpf} if $u\in \cM$. The space of maps $\cM$ is manifestly irreducible, so we may speak of a general point $f\in \cM$.

Let $\zeta_1\in A^1(Q)$ be the pullback of the hyperplane class on $\bP(H^0(\bP^1,\cO(d-k))^{r-s})$, and let $\zeta_2\in A^1(Q)$ be the relative hyperplane class.

\subsection{Incidence loci}

Recall that we are interested in the loci of maps $f:\bP^1\to X_{r,s}$ satisfying conditions $f(p_i)\in \Lambda_i$, where $\Lambda_i\subset X_{r,s}$ is a linear subspace. In this section, we consider the compactification of this locus inside the quasimap space $Q$.

For a point $p\in \bP^1$ and a linear space $\Lambda\in \Gamma(\ell_1,\ell_2)$, define the \emph{incidence locus} $\Inc(p,\Lambda)\subset Q$ by the same linear equations (in the variables $u_j$) which cut out $\Lambda$ (in the variables $x_j$). Upon restriction to $\cM$, we have $\Inc(p,\Lambda)=\ev_p^{-1}(\Lambda)$ as subschemes.

\begin{proposition}\label{prop:inc_class}
For any $\Lambda\in\Gamma(\ell_1,\ell_2)$, we have
\begin{equation*}
[\Inc(p,\Lambda)]=\zeta_1^{(r-s-1)-\ell_1}\zeta_2^{(s+1+\ell_1)-\ell_2}.
\end{equation*}
\end{proposition}

It is a special case of Proposition \ref{prop:main_transversality} that $\Inc(p,\Lambda)$ has expected codimension below, and in fact that it is the closure of the locus of maps $f$ with $f(p)\in\Lambda$. In particular, the meaning of the cycle class $[\Inc(p,\Lambda)]$ is umambiguous. 

\begin{proof}
The locus on $Q$ where a linear combination of $u_1,\ldots,u_{r-s}$ vanishes at $p$ is the degeneracy locus of a map
\begin{equation}\label{eq:degen1}
\cO_{ \bP(H^0(\bP^1,\cO_{\bP^1}(d-k))^{r-s})}(-1) \to H^0(\bP^1,\cO(d-k)|_p),
\end{equation}
and the target line bundle is trivial. Similarly, the locus on $Q$ where a linear combination of $u_0u_1,\ldots,u_{r+1}$ vanishes at $p$ is the degeneracy locus of a map
\begin{equation}\label{eq:degen2}
\cO_{Q}(-1) \to H^0(\bP^1,\cO(d)|_p).
\end{equation}

The subvariety $\Inc(p,\Lambda)\subset Q$ is cut out by $(r-s-1)-\ell_1$ maps of the form \eqref{eq:degen1} and $(s+1+\ell_1)-\ell_2$ of the form \eqref{eq:degen2}, so the cycle class formula follows.
\end{proof}

\subsection{Transversality}

In this section, we show that general incidence loci $\Inc(p_i,\Lambda_i)\subset Q$ always intersect in the expected dimension, and generically in the open locus $\cM\subset Q$ of maps.

\begin{situation}\label{situation}
Let $p_1,\ldots,p_n\in\bP^1$ be distinct points.

Let $\Lambda_1,\ldots,\Lambda_n\subset X_{r,s}$ be a general collection of linear spaces. Write $(\ell^i_1,\ell^i_2)$ for the bi-dimension of $\Lambda_i$. Write also
\begin{align*}
(m^i_1,m^i_2)=((r-s-1)-\ell^i_1,(s+1+\ell^i_1)-\ell^i_2)
\end{align*}
and $m^i=m^i_1+m^i_2$ for the codimension of $\Lambda_i$. Write
\begin{align*}
m&=\sum_i m^i,\\
m_j&=\sum_i m^i_j,\text{ for }j=1,2.
\end{align*}
Define
\begin{equation*}
I:=\bigcap_{i=1}^{n}\Inc(p_i,\Lambda_i)\subset Q.
\end{equation*}
\end{situation}

\begin{proposition}\label{prop:main_transversality}
In Situation \ref{situation}, the intersection $I$ is pure of codimension $m$, the expected. Furthermore, every general point of $I$ is contained in $\cM$.

\end{proposition}

Proposition \ref{prop:main_transversality} is proven in two parts, by restricting separately to $\cM$ and its complement.

\begin{proposition}\label{prop:transversality_M}
Let $I$ be as in Situation \ref{situation}. Then, $I\cap \cM\subset \cM$ is pure of codimension $m$.
\end{proposition}

In fact, in the proof below, the points $p_i\in\bP^1$ need not be distinct.

\begin{proof}
Consider the map
\begin{equation*}
\ev:\cM\times\prod_{i=1}^n \Gamma(\ell^i_1,\ell^i_2) \to \prod_{i=1}^n X\times \Gamma(\ell^i_1,\ell^i_2) 
\end{equation*}
given by $\prod_{i=1}^{n}\ev_i$. Let $Z=\ev^{-1}(\prod_{i=1}^{n}\Lambda^{\univ}_i)$. 

Consider the projection $\pr:Z\to \cM$. The fiber over $f\in \cM$ is the locus of linear spaces 
\begin{equation*}
(\Lambda_1,\ldots,\Lambda_n)\in \prod_{i=1}^n \Gamma(\ell^i_1,\ell^i_2)
\end{equation*}
containing the points $f(p_1),\ldots,f(p_n)$. By Lemma \ref{lem:point_imposes_expected_conditions_on_Gamma}, the fibers of $\pr$ are non-empty and smooth of dimension
\begin{equation*}
\dim\left(\prod_{i=1}^n\Gamma(\ell^i_1,\ell^i_2)\right)-m,
\end{equation*}
hence $Z$ is smooth of dimension
\begin{equation*}
\dim(Q)+\dim\left(\prod_{i=1}^n\Gamma(\ell^i_1,\ell^i_2)\right)-m.
\end{equation*}
Now, the intersection $I$ is a general fiber of the other projection $Z\to \prod_{i=1}^n\Gamma(\ell^i_1,\ell^i_2)$, so the Proposition follows by comparing dimensions of the source and target.
\end{proof}

\begin{proposition}\label{prop:transversality_twisting}
Let $I$ be as in Situation \ref{situation}. Then, no irreducible component of $I$ is supported in $Q\backslash \cM$.
\end{proposition}

\begin{proof}
For sake of contradiction, let $J\subset I$ be such a component, and let $u\in J$ be a general point, and write $u_j$ for the sections underlying $u$. We will estimate the codimension of $J$ by counting parameters. That the conditions enumerated are independent, so that the codimension estimate is correct, follows from an incidence correspondence argument as in the previous case.

First, assume that the section $u_0$ is not identically zero. We apply the twisting operations of \cite[\S 2.5]{cl_complete} and \cite[\S 3.2]{lsakran}. Namely, consider the operations
\begin{enumerate}
\item[(T1)] If (BPF1) fails at some $p\in\bP^1$, then twist each of $u_1,\ldots,u_{r-s}$ down by $p$, and twist $u_0$ up by $p$, thereby increasing $k$ by 1 and keeping $d$ constant.
\item[(T2)] If (BPF2) fails at some $p\in\bP^1$, then twist each of $u_0,u_{r-s+1},\ldots,u_{r+1}$ down by $p$, thereby decreasing both $d$ and $k$ by 1.
\end{enumerate}
Then, apply \cite[Algorithm 3.2.2]{lsakran} to $u$. Note that these operations make sense even if some of the sections $u_j$ are identically zero for $j>0$. In contrast to the settings of \cite[\S 2.5]{cl_complete} and \cite[\S 3.2]{lsakran}, we do not keep track of any data of effective divisors along which the $u_j$ vanish. However, we assume that $u_0\neq0$ to ensure that operation (T2) always leaves the value of $k$ non-negative.

Every twist strictly decreases the difference between the number of moduli for $Q$ (as a function of $d,k$) and number of conditions retained by $\Inc(p_i,\Lambda_i)$. Indeed, (T1) decreases number of moduli by $r-s$ and destroys at most $r-s-1$ conditions imposed by $\Inc(p_i,\Lambda_i)$. (T2) decreases number of moduli by $s+2$ and destroys at most $s+1$ conditions imposed by $\Inc(p_i,\Lambda_i)$. [Obviously applying a twist at $q\neq p_i$ only decreases moduli.] Then, apply Proposition \ref{prop:transversality_M}, which is valid for the new values of $d,k$ after twisting because it remains the case that $d\ge k\ge 0$.\footnote{If $k<0$, but $d\ge0$, then the moduli space of quasimaps of degree $\beta$ is well-defined, but in general has larger than expected dimension, and a general point may not correspond to a map $f:\bP^1\to X_{r,s}$.}

It remains to consider the case in which $u_0=0$ identically. The locus on $Q$ where $u_0=0$ has codimension $k+1$, and is identified with the moduli space 
\begin{equation*}
Q_0:=Q_{d-k}(\bP^1,\bP^{r-s-1})\times Q_{d}(\bP^1,\bP^{s})=\bP(H^0(\cO_{\bP^1}(d-k))^{r-s}) \times \bP(H^0(\cO_{\bP^1}(d))^{s+1}) 
\end{equation*}
quasimaps of bi-degree $(d-k,d)$ from $\bP^1$ to $\bP^{r-s-1}\times \bP^{s}$. The incidence loci $\Inc(p_i,\Lambda_i)$ pull back on $Q_0$ to products of incidence loci $\Inc(p_i,\Lambda^1_i)\times\Inc(p_i,\Lambda^2_i)$. Here, $\Lambda^1_i$ is the image of $\Lambda_i$ under $\pi$, and $\Lambda^2_i$ is the linear subspace of $\bP^{s}$ corresponding to the remaining $s+1+\ell_1-\ell_2$ equations cutting out $\Lambda_i$, now in the sections $u_{r-s+1},\ldots,u_{r+1}$. If $\ell_1=\ell_2$, then $\Lambda^2_i$ is taken to be empty, and $\Inc(p_i,\Lambda^2_i)$ is simply the locus where
\begin{equation*}
[u_{r-s+1}:\cdots:u_{r+1}]\in Q_{d}(\bP^1,\bP^{s})
\end{equation*}
has a base-point at $p_i$. The codimension of $\Inc(p_i,\Lambda_i)$ is preserved under pullback.

Repeating now the previous arguments of Proposition \ref{prop:transversality_M} and the twisting of base points above to $Q_{d-k}(\bP^1,\bP^{r-s-1})$ and $Q_{d}(\bP^1,\bP^{s})$ separately, we see that the pullback of the intersection $I$ to $Q_0$ has the expected codimension of $m$ in a neighborhood of $u$. The conclusion follows, because $Q_0\subset Q$ has positive codimension.
\end{proof}

Propositions \ref{prop:transversality_M} and \ref{prop:transversality_twisting} together imply Proposition \ref{prop:main_transversality}.

\begin{proposition}\label{prop:reduced}
    In Situation \ref{situation}, suppose one of the following assumptions holds.
    \begin{enumerate}
        \item[(i)] $\charac(\bK)=0$, or
        \item[(ii)] every $\Lambda_i$ is the proper transform of a linear subspace of $\bP^r$ that is transverse to the blow-up center $\bP^s$. Equivalently, $m_2^i=\min(s+1,m^i)$. 
    \end{enumerate}
    Then, $I$ is generically reduced.
\end{proposition}

For example, the hypothesis of (ii) is satisfied for points in $X_{r,s}$.

\begin{proof}
    If $\charac(\bK)=0$, then generic reducedness of $I$ follows from generic smoothness applied to the morphism $Z\to \prod_{i=1}^n\Gamma(\ell^i_1,\ell^i_2)$ from the proof of Proposition \ref{prop:transversality_M}, and the fact, by Proposition \ref{prop:transversality_twisting}, that $I$ lies generically in $\cM$. This proves (i).

    In the setting of (ii), let $f\in I$ be a point, necessarily a point of $\cM$. The space of deformations of $f$ preserving the conditions $f(p_i)\in \Lambda_i$ is
    \[
    H^0(\bP^1,f^*T_{X_{r,s}}[p_1\xrightarrow[]{-}T_{f(p_1)}\Lambda_1]\cdots [p_n\xrightarrow[]{-}T_{f(p_n)}\Lambda_n]).
    \]
    Under the assumption on the $\Lambda_i$, the $\Lambda_i\subset T_{(p_i)}X_{r,s}$ are \emph{general} subspaces. Therefore, the modifications can be taken in general directions. By the proof of Lemma \ref{lemma:ev_is_surjective}, all summands of $f^*T_{X_{r,s}}$ have positive degree. If $I$ is non-empty, then $m\le \dim(Q)$, and the degree of the modified bundle
    \[
    f^*T_{X_{r,s}}[p_1\xrightarrow[]{-}T_{p_1}\Lambda_1]\cdots [p_n\xrightarrow[]{-}T_{p_n}\Lambda_n]
    \]
    is at least $-r$. Because the modifications are in general directions, the modifications can be arranged in such a way that all summands of the modified bundle have degree at least $-1$.
\end{proof}

\subsection{Invariants}

\begin{proposition}\label{prop:integral_value}
If $\dim(Q)=m_1+m_2$, then
\begin{equation*}
\int_Q \zeta_1^{m_1}\zeta_2^{m_2}=
\begin{cases}
\binom{(r-s)(d-k+1)-1-m_1+k}{k}\text{ if }m_1\le (r-s)(d-k+1)-1,\\
0\text{ otherwise}
\end{cases}
.
\end{equation*}
\end{proposition}

\begin{proof}
We identify $Q$ with the projective bundle $\pi:\bP(\cO(-1)^{k+1}\oplus\cO^{(d+1)(s+1)})\to \bP^{(r-s)(d-k+1)-1}$. If $m_1>(r-s)(d-k+1)-1$, then $\zeta_1^{m_1}=0$. Otherwise,
\begin{align*}
\int_Q \zeta_1^{m_1}\zeta_2^{m_2}&=\int_{ \bP^{(r-s)(d-k+1)-1}} \zeta_1^{m_1}\cdot s(\cO(-1)^{k+1}\oplus\cO^{(d+1)(s+1)})\\
&=\int_{ \bP^{(r-s)(d-k+1)-1}} \zeta_1^{m_1}\cdot (1-\zeta_1)^{-(k+1)}\\
&=(-1)^{(r-s)(d-k+1)-1-m_1}\binom{-(k+1)}{(r-s)(d-k+1)-1-m_1}\\
&=\binom{(r-s)(d-k+1)-1-m_1+k}{k}.
\end{align*}
\end{proof}

Combining Propositions \ref{prop:main_transversality} and \ref{prop:integral_value}, we obtain:

\begin{corollary}\label{cor:map_existence}
Suppose that we are in Situation \ref{situation}, that $m=\dim(Q)$, and that 
\begin{equation*}
m_1\le (r-s)(d-k+1)-1.
\end{equation*}
Then, the set of maps $f:\bP^1\to X_{r,s}$ of degree $\beta$ with $f(p_i)\in\Lambda_i$ for all $i$ is finite and non-empty.
\end{corollary}

\begin{proof}
    Proposition \ref{prop:main_transversality} shows that the integral of Proposition \ref{prop:integral_value} computes the length of the zero-dimensional subscheme $I\subset \cM$, parametrizing the maps in question. In particular, under the stated hypothesis, $I$ is non-empty. 
\end{proof}

Corollary \ref{cor:map_existence} gives an alternative proof of Corollary \ref{cor:interpolation_tangent}. (In fact, in characteristic zero, one can also re-compute $T_{d,k}$ only from the data of when the integral of Proposition \ref{prop:integral_value} is non-zero, by varying the bi-dimensions of the input linear spaces.)

The (set-theoretic) \emph{number} of interpolating maps in Question \ref{q:enumerative} and Corollary \ref{cor:map_existence} is equal to the value of the integral of Proposition \ref{prop:integral_value} whenever $I$ is reduced, e.g. under the assumptions of Proposition \ref{prop:reduced}. It is not immediate to us whether the reducedness holds in more generality. However, in the setting of Theorem \ref{thm:tev}, in which the $\Lambda_i$ are points, Corollary \ref{cor:map_existence} computes the set-theoretic number of interpolating maps in arbitrary characteristic. In particular, we obtain Theorem \ref{thm:tev} set-theoretically in arbitrary characteristic.

\bibliographystyle{alpha} 
\bibliography{interpolation_secant.bib}

\end{document}